\documentclass[11pt]{article}
\usepackage[utf8]{inputenc} % (!) 1er (ou 2e) package
\usepackage[T1]{fontenc} % (!) 2e (ou 1er) package
\usepackage{amsmath}
\usepackage{amsthm}
\usepackage{amsfonts}
\usepackage{amssymb}
\usepackage{lmodern} % ordre indifférent
\usepackage{tikz}
\usetikzlibrary{arrows}

\usepackage{pgf,tikz,pgfplots}
\usepackage{mathrsfs}

\definecolor{ffqqqq}{rgb}{1,0,0}
\definecolor{wrwrwr}{rgb}{0.3803921568627451,0.3803921568627451,0.3803921568627451}
\definecolor{ffqqqq}{rgb}{1,0,0}
\definecolor{wrwrwr}{rgb}{0.3803921568627451,0.3803921568627451,0.3803921568627451}
\definecolor{rvwvcq}{rgb}{0.08235294117647059,0.396078431372549,0.7529411764705882}
\definecolor{cqcqcq}{rgb}{0.7529411764705882,0.7529411764705882,0.7529411764705882}
\definecolor{ffffff}{rgb}{1,1,1}
\definecolor{ccqqqq}{rgb}{0.8,0,0}

\usepackage[a4paper]{geometry}
\usepackage{cite}
\usepackage[english]{babel}

\newtheorem{theorem}{Theorem}[section]
\newtheorem{lemma}[theorem]{Lemma}

\newtheorem{corollary}[theorem]{Corollary}

\newtheorem{remark}[theorem]{Remark}

\newtheorem{question}[theorem]{Question}

\newcommand{\N}{{\mathbb N}}
\newcommand{\Hyp}{{\mathbb H}}
\newcommand{\Hy}{{\mathbb H}}
\newcommand{\Z}{{\mathbb Z}}
\newcommand{\R}{{\mathbb R}}
\newcommand{\Q}{{\mathbb Q}}

\newcommand{\T}{{\mathbb T}}

\newcommand{\Int}{\mbox{Int}}
\newcommand{\Vol}{\mbox{Vol}}

\begin{document}

\title{Algebraic intersection for translation surfaces in a family of  Teichm\H{u}ller disks}
\author{ S. Cheboui, A. Kessi, D. Massart}
\date{\today}
\maketitle
\begin{abstract}
We give a hyperbolic-geometric construction to compute the quantity KVol defined in Equation (\ref{defKVol}) in a family of Teichm\H{u}ller disks of square-tiled surfaces.
\end{abstract}
\section{Introduction}
\subsection{Definitions}
Let $X$  be a closed surface, that is, a compact, connected manifold of dimension 2, without boundary. Let us assume that $X$ is oriented. If two $C^1$ closed curves $\alpha$ and $\beta$ in $X$ intersect transversally at a point   $P \in X$, we set $\mbox{Int}_P (\alpha, \beta)=1$ if $\beta$ crosses $\alpha$ from right to left, and $\mbox{Int}_P (\alpha, \beta)=-1$ otherwise. Then the algebraic intersection 
$\mbox{Int} (\alpha, \beta)$ of $\alpha$ and $\beta$ is the sum over all intersection points $P$ of  $\mbox{Int}_P (\alpha, \beta)$. The algebraic intersection endows the first homology $H_1(X,\R)$ with a symplectic bilinear form. In particular $\mbox{Int} (\alpha, \beta)$ is finite, and only depends on the homology classes of $\alpha$ and $\beta$.

Now let us assume $X$ is endowed with a Riemannian metric $g$. We denote $\mbox{Vol}(X,g)$ the Riemannian volume of $X$ with respect to the metric $g$, and for any piecewise smooth closed curve $\alpha$ in $X$, we denote $l_g(\alpha)$ the length of $\alpha$ with respect to $g$. When there is no ambiguity we omit the reference to $g$. 

We are interested in the quantity
\begin{equation}\label{defKVol}
\mbox{KVol}(X,g) = \Vol (X,g) \sup_{\alpha,\beta} \frac{\Int (\alpha,\beta)}{l_g (\alpha) l_g (\beta)}
\end{equation}
where the supremum ranges over all piecewise smooth closed curves $\alpha$ and $\beta$ in $X$. The $\Vol (X,g)$ factor is there to make $\mbox{KVol}$ invariant to re-scaling of the metric $g$. See \cite{MM} as to why $\mbox{KVol}$ is finite. The quantity KVol comes up naturally when you want to compare the stable norm (a norm which measures the length of a homology class, with respect to the metric $g$) with the Hodge norm (or $L^2$-norm) on $H_1 (M,\R)$ (see \cite{MM}). It is easy to make $\mbox{KVol}$ go to infinity, you just need to pinch a non-separating closed curve $\alpha$ to make its length go to zero. The interesting surfaces are those $(X,g)$ for which $\mbox{KVol}$ is small. 

When $X$ is the torus, we have $\mbox{KVol} (X,g) \geq 1$, with equality if and only if the metric $g$ is flat (see \cite{MM}). Furthermore, when $g$ is flat, the supremum in (\ref{defKVol}) is not attained, but for a negligible subset of the set of all flat metrics.   In \cite{MM} $\mbox{KVol}$ is studied as a function of $g$, on the moduli space of hyperbolic  (that is, the curvature of $g$ is $-1$) surfaces of fixed genus. It is proved that $\mbox{KVol}$ goes to infinity when $g$ degenerates by pinching a non-separating closed curve, while $\mbox{KVol}$ remains bounded when $g$ degenerates by pinching a separating closed curve.

This leaves open the question whether $\mbox{KVol}$ has a minimum over the moduli space of hyperbolic surfaces of genus $n$, for $n \geq 2$. It is conjectured in \cite{MM} that for almost every $(X,g)$ in the moduli space of hyperbolic surfaces of genus $n$, the supremum in (\ref{defKVol})  is  attained (that is, it is actually a maximum). 

In this paper we consider a different class of surfaces : translation surfaces of genus $s$, with one conical point. The set (or stratum) of such surfaces is denoted $\mathcal{H}(2s-2)$ (see \cite{HL}). 
We consider the family of  translation surfaces $St(2s-1)$ (so named after \cite{Schmithusen}) depicted in Figure \ref{L22}, for $s \in \N, s \geq 2$,  obtained by gluing the opposite sides of a staircase-shaped template made of $2s-1$ squares (see Figure \ref{L22}).

\begin{figure} [h!]
\begin{center}

\begin{tikzpicture}[scale=2]
\draw (0,0) -- (2,0);
\draw (0,0) -- (0,1);
\draw (1,0) -- (1,2);
\draw (2,0) -- (2,2);
\draw (0,1) -- (2,1);
\draw (1,0) -- (1,1);
\draw (1,2) -- (2,2);
\draw (0,0) node {$\bullet$} ;
\draw (1,0) node {$\bullet$} ;
\draw (2,0) node {$\bullet$} ;
\draw (0,1) node {$\bullet$} ;
\draw (1,1) node {$\bullet$} ;
\draw (2,1) node {$\bullet$} ;
\draw (2,2) node {$\bullet$} ;
\draw (1,2) node {$\bullet$} ;
\draw (0.5,0) node {$\backslash $};
\draw (0.5,1) node {$\backslash $};
\draw (1.5,0) node {$|$};
\draw (1.5,2) node {$|$};
\draw (0,0.5) node {$- $};
\draw (2,0.5) node {$-$};
\draw (2,1.5) node {$=$};
\draw (1,1.5) node {$=$};

\draw (3,0) -- (5,0);
\draw (3,1) -- (6,1);
\draw (4,2) -- (6,2);
\draw (5,3) -- (6,3);
\draw (3,0) -- (3,1);
\draw (4,0) -- (4,2);
\draw (5,0) -- (5,3);
\draw (6,1) -- (6,3);
\draw (3,0) node {$\bullet$} ;
\draw (4,0) node {$\bullet$} ;
\draw (5,0) node {$\bullet$} ;
\draw (3,1) node {$\bullet$} ;
\draw (4,1) node {$\bullet$} ;
\draw (5,1) node {$\bullet$} ;
\draw (6,1) node {$\bullet$} ;
\draw (4,2) node {$\bullet$} ;
\draw (5,2) node {$\bullet$} ;
\draw (6,2) node {$\bullet$} ;
\draw (5,3) node {$\bullet$} ;
\draw (6,3) node {$\bullet$} ;
\draw (3.5,0) node {$\backslash $};
\draw (3.5,1) node {$\backslash $};
\draw (4.5,0) node {$|$};
\draw (4.5,2) node {$|$};
\draw (3,0.5) node {$- $};
\draw (5,0.5) node {$-$};
\draw (4,1.5) node {$=$};
\draw (6,1.5) node {$=$};
\draw (5.5,1) node {$\backslash \backslash$};
\draw (5.5,3) node {$\backslash \backslash$};
\draw (5, 2.5) node {x};
\draw (6, 2.5) node {x};
\end{tikzpicture}
\caption{ $St(3)$  and $St(5)$} \label{L22}
\end{center}
\end{figure}
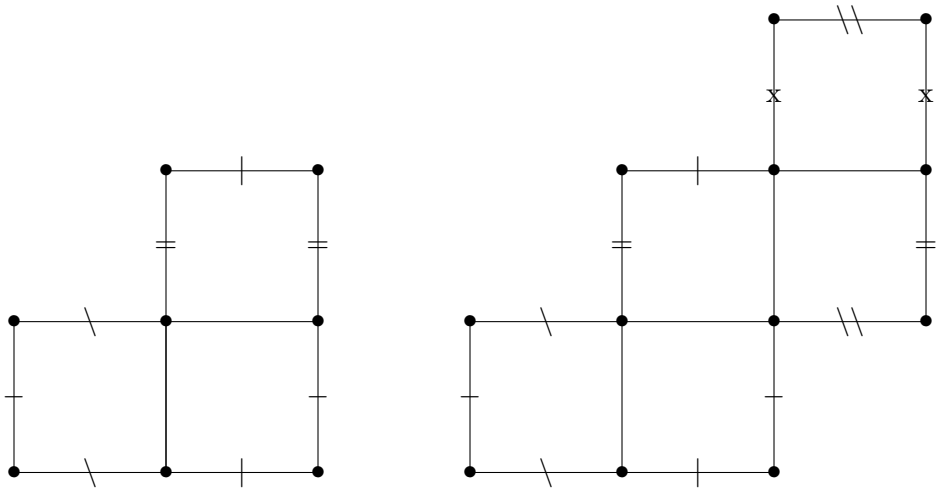

Our first result is (see Corollary \ref{corL22})
\[
\forall s \geq 2, \  \mbox{KVol}(St(2s-1))=2s-1.
\]
This is the first exact computation of $\mbox{KVol}$, outside of flat tori.

Before stating our next result we need to elaborate a little bit on the Teichm\H{u}ller disk of $St(2s-1)$.
%%%%%%%%
\subsection{The Teichm\H{u}ller disc $\mathcal{T}$ of $St(2s-1)$}
Let us explain the terminology. Every translation surface may be viewed as a plane template with parallel sides of equal length pairwise identified. The group $GL_2^+ (\R )$ acts linearly on templates, preserving identifications, so it acts on translation surfaces. It may happen that for some template $T$ and some element $A$ of $GL_2^+ (\R )$, both $T$ and $A.T$ are templates of the same translation surface $X$. Then we say that $A$ lies in the Veech group of $X$, that is, the subgroup of $GL_2^+ (\R )$  which preserves $X$. Since the Veech group must preserve volumes, it is a subgroup of $SL_2 (\R )$, and it turns out to be a Fuchsian group (see \cite{Veech}). The orbit of $X$ under $GL_2^+ (\R )$ is called the Teichm\H{u}ller disk of $X$. For the purpose of studying $\mbox{KVol}$, it is convenient to identify templates which are related by a similitude transformation (an isometry composed with a re-scaling). Recall that $GL_2^+( \R )$ quotiented by similitudes is the hyperbolic plane $\Hyp ^2$. The Teichm\H{u}ller disk of $X$ may thus be seen as the quotient of $\Hyp ^2$ by the Veech group of $X$, which is a Fuchsian group.

 We denote by $\mathcal{T}(St(2s-1))$, or just $\mathcal{T}_s$ if there is no ambiguity, the Teichm\H{u}ller disc $\mathcal{T}$ of $St(2s-1)$.
The Veech group $\Gamma$ of $St(2s-1)$ is generated, for any $s \geq 2$,   by 
\[
T= \left(
\begin{array}{cc}
1&2\\
0&1
\end{array}
\right)
\mbox{ and }
R= \left(
\begin{array}{cc}
0&-1\\
1&0
\end{array}
\right),
\]
it is a subgroup of index $3$ in $SL_2 (\Z )$ (see \cite{Schmithusen2}, or \cite{D}). It has a fondamental domain $\mathcal{D}$ in $\Hyp ^2$ comprised between the straight lines $x=\pm 1$, and the unit half-circle centered at the origin. The two vertical straight lines are identified by $T$, and the two halves of the unit half-circle are identified by $R$, with $i$ as fixed point (see Figure \ref{disqueL22}).
\begin{figure} [h!]
\begin{center}

\begin{tikzpicture}[line cap=round,line join=round, scale=2, >=triangle 45,x=1cm,y=1cm]
%\clip(-1.2688205410698812,-0.3157475471986166) rectangle (0.7586299339180109,1.8323068949509285);
\draw [line width=1pt] (1,-0.3157475471986166) -- (1,1.8323068949509285);
\draw [line width=1pt] (-1,-0.3157475471986166) -- (-1,1.8323068949509285);
%\draw [shift={(0,0)},line width=1pt,color=ffffff,fill=ffffff,fill opacity=0.1]  (0,0) --  plot[domain=2.085424903227732:3.141592653589793,variable=\t]({1*1*cos(\t r)+0*1*sin(\t r)},{0*1*cos(\t r)+1*1*sin(\t r)}) -- cycle ;
%\draw [shift={(0,0)},line width=1pt,color=ffffff,fill=ffffff,fill opacity=0.1]  (0,0) --  plot[domain=2.085424903227732:3.141592653589793,variable=\t]({1*1*cos(\t r)+0*1*sin(\t r)},{0*1*cos(\t r)+1*1*sin(\t r)}) -- cycle ;
\draw [shift={(0,0)},line width=1pt]  plot[domain=0:3.141592653589793,variable=\t]({1*1*cos(\t r)+0*1*sin(\t r)},{0*1*cos(\t r)+1*1*sin(\t r)});
\begin{scriptsize}
\draw [fill=ccqqqq] (1,0) circle (1pt);
\draw[color=ccqqqq] (1.1259144016816881,0.07051110438669614) node {$1$};
%\draw[color=black] (0.7488512338778764,1.8176388448907268) node {$f$};
%\draw[color=black] (-0.9624212731456661,1.8176388448907268) node {$g$};
\draw [fill=ccqqqq] (0,1) circle (1pt);
\draw[color=ccqqqq] (0.035006130948055846,1.1) node {$i$};
\draw [fill=ccqqqq] (-1,1) circle (1pt);
\draw[color=ccqqqq] (-1.3,1.07119807516046) node {$- 1 +i$};
\draw [fill=ccqqqq] (-0.49221156784221043,0.8704756013136228) circle (1pt);
\draw[color=ccqqqq] (-0.6,0.7) node {$\exp (2i \pi/3)$};
\draw [fill=ccqqqq] (1,1) circle (1pt);
\draw[color=ccqqqq] (1.2259144016816881,1.07119807516046) node {$ 1 +i$};
\draw [fill=ccqqqq] (0.5007280546607461,0.8656046529886867) circle (1pt);
\draw[color=ccqqqq] (0.6,0.7) node {$\exp (i \pi/3)$};
\draw [fill=ccqqqq] (-1,0) circle (1pt);
\draw[color=ccqqqq] (-1.1754595398658453,0.06399197102660648) node {$-1$};
\draw [fill=ccqqqq] (0,0) circle (1pt);
\draw[color=ccqqqq] (0.025227430907921317,0.16399197102660648) node {$0$};
%\draw[color=ffffff] (-0.8320386059438722,0.5627056730734659) node {$k$};
%\draw[color=ffffff] (-0.8320386059438722,0.5627056730734659) node {$p$};
\end{scriptsize}
\end{tikzpicture}
\caption{A fundamental domain for the Teichm\H{u}ller disk of $St(2s-1)$}
\label{disqueL22}
\end{center}
\end{figure}
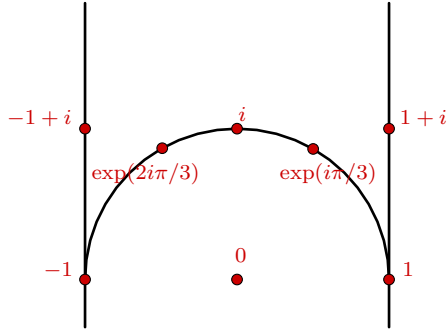

Here is a convenient way to view the Teichm\H{u}ller disk  $\mathcal{T}_s$ of $St(2s-1)$. Any element  $X$ of $\mathcal{T}_s$, such as the one depicted on the right in  Figure \ref{template}, has a template made of $2s-1$ congruent parallelograms.  Modulo some similitude, we may assume b=1. Then $\beta$ lies in the fundamental domain depicted above, and completely determines $X$, together with the number $2s-1$ of squares. We shall often label  the surface $X$ by the complex number $\alpha$, or the corresponding vector in $\R^2$, and the integer $s$. For instance,  $St(3)$ is $X_2(i)=X_2(0,1)$, and $St(2s-1)=X_s(i)$ for all $s$. There are other interesting surfaces in   $\mathcal{T}_2$  :  $X_2(\pm 1 +i)=X_2(\pm 1,1)$ is the only (other than $St(3)$) three-square surface of genus $2$,
\begin{figure}[h!]
\begin{center}
\begin{tikzpicture}
\draw [line width=1pt] (-5,0)-- (-5,3)-- (-4,3)-- (-4,0);
\draw [line width=1pt] (-5,0)-- (-4,0);
\draw [line width=.5pt,dash pattern=on 1pt off 1pt] (-5,1)--(-4,1);
\draw [line width=.5pt,dash pattern=on 1pt off 1pt] (-5,2)--(-4,2);

\draw (-4.73,0) node[anchor=north west] {b};
\draw (-4.69,3.56) node[anchor=north west] {b};
\draw (-5.83,3.06) node[anchor=north west] {$\beta$};
\draw (-5.75,2.1) node[anchor=north west] {$\alpha$};
\draw (-5.67,0.94) node[anchor=north west] {$\gamma$};
\draw (-3.97,3.04) node[anchor=north west] {$\gamma$};
\draw (-3.97,2.04) node[anchor=north west] {$\alpha$};
\draw (-3.93,1.02) node[anchor=north west] {$\beta$};

\draw [line width=0.8pt] (2,0)-- (6,0);
\draw [line width=0.8pt] (2,0)-- (4,3.4641016151377544);
\draw [line width=0.8pt] (3,1.7320508075688772)-- (7,1.7320508075688772);
\draw [line width=0.8pt] (6,0)-- (7,1.7320508075688772);
\draw [line width=0.8pt] (4,0)-- (5,1.7320508075688772);
\draw [line width=0.8pt] (4,3.4641016151377544)-- (6,3.4641016151377544);
\draw [line width=0.8pt] (5,1.7320508075688772)-- (6,3.4641016151377544);
\draw (2.49,-0.02) node[anchor=north west] {b};
\draw (4.73,0) node[anchor=north west] {a};
\draw (4.61,4.04) node[anchor=north west] {b};
\draw (5.5,2.4) node[anchor=north west] {a};
\draw (2.7,3.26) node[anchor=north west] {$\alpha$};
\draw (1.7,1.52) node[anchor=north west] {$\beta$};
\draw (5.69,3.16) node[anchor=north west] {$\alpha$};
\draw (6.79,1.54) node[anchor=north west] {$\beta$};
\end{tikzpicture}
\caption{$X_2(1,1)$ and $X_2(\exp (i \pi /3))$}
\label{template}
\end{center}
\end{figure}
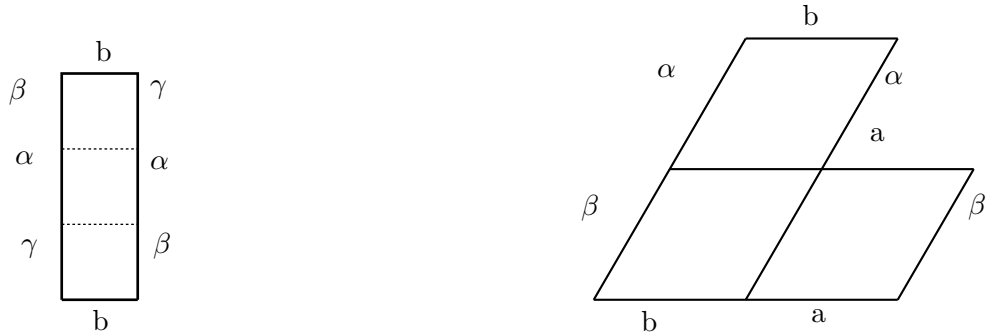
while $X_2(\exp (i \pi /3))=X_2(\exp (2i \pi /3))$ is the only translation surface of genus $2$ tiled by 6 equilateral triangles. It turns out that $X_s(1,1)$ and $X_s(\exp (i \pi /3))$ also have KVol$=2s-1$.

For $(x,y) \in \Hy^2$, we denote $T(x,y)$ the flat torus
\[
\R^2 / \Z (1,0) \oplus \Z (x,y).
\]
The translation surface $X_s(x,y)$ is a ramified, $(2s-1)$-fold, Riemannian cover of $T(x,y)$. For $(p,q) \in \Z^2$ with $p \wedge q =1$, we say a  geodesic segment in $T(x,y)$ has direction $(p,q)$ if its lifts to the universal cover $\R^2$ of $T(x,y)$ are parallel to the vector $p(1,0)+q(x,y)$. We say a  geodesic segment  in $X_s(x,y)$ has direction $(p,q)$ if it projects to a geodesic segment of direction $(p,q)$ in $T(x,y)$.

Beware this may cause a bit of confusion at first, because a geodesic segment in $T(x,y)$ with direction $(p,q)$ may not lift, in $\R^2$, to a segment parallel to $(p,q)$, unless $(x,y)=(0,1)$. The upside is that this definition is independent of $(x,y)$. For instance, the algebraic intersection of two simple closed geodesics in $T(x,y)$, of respective directions $(p,q)$ and $(p',q')$, is $pq'-p'q$, regardless of $(x,y)$.

This allows a neat formulation of the following property : given $p,q,p',q' \in \Z$ with $p\wedge q = p'\wedge q'=1$, and $p/q \neq p'/q'$ in $\Q \cup \{ \infty \}$, the hyperbolic geodesic with endpoints $p/q$ and $p'/q'$ is the locus of the flat tori in which the geodesics with respective directions $(p,q)$ and $(p',q')$ are orthogonal. Since the covering $X(x,y) \longrightarrow T(x,y)$ is Riemannian, the hyperbolic geodesic with endpoints $p/q$ and $p'/q'$ is also the locus of surfaces $X(x,y)$ in which the geodesics with respective directions $(p,q)$ and $(p',q')$ are orthogonal. 

In the torus all directions are alike, because the group $SL_2 (\Z ) $ of orientation-preserving, affine diffeomorphims of $\T^2$ acts transitively on the set of directions. However (see \cite{HL}), the Veech group  $\Gamma$ of $St(2s-1)$ has two orbits, that of $(0,1)$, which comprises all directions $(p,q)$ with $p \neq q \mod 2$, and that of $(1,1)$, which comprises all directions $(p,q)$ with $p = q= 1 \mod 2$. 
Note that $(p,q)=(1,0)$ gives $p/q = \infty$, which is one of the points at infinity of the Teichm\H{u}ller disk $\mathcal{T}$, while $(1,1)$ gives $1$, which is the other point at infinity of $\mathcal{T}$.

Simple closed geodesics on a translation surface in $\mathcal{H}(2s-2)$ are of two kinds  (see \cite{HL}): saddle connections, which go through the conical point, and non-singular closed geodesics, which do not. Since our surfaces have only one conical point, every closed curve is homologous to a linear combination of closed saddle connections, so in the definition of KVol we may consider only closed saddle connections. Non-singular closed geodesics come in cylinders of parallel geodesics of equal length. In a translation surface in $\mathcal{H}(2s-2)$, for any direction, there are at most $s$ such cylinders. If the translation surface lies in $\mathcal{T}_s$, the $s$-cylinder directions are precisely those in the orbit, under the Veech group $\Gamma$, of $(0,1)$, and the other directions have only one cylinder. 

For any direction $(p,q)$, denote  $r= p/q$ and $l_r (x,y)$ the length, in the torus $T(x,y)$, of simple closed geodesics of direction  $(p,q)$. The horocycles with point at infinity $r$ are the level sets of the function 
$l_r$. We shall see (Lemma \ref{longueur}) that the saddle connections, in $X(x,y)$, of direction $(p,q)$, have the same length $l_r (x,y)$. 

Therefore there are two ways of going to infinity in the Teichm\H{u}ller disk $\mathcal{T}_s$ : $(x,y)$ may converge to $(1,0)$, in which case there is a one-cylinder direction in $X(x,y)$ whose saddle connections become arbitrarily short ; or $y \longrightarrow \infty$, in which case there is an $s$-cylinder direction in $X(x,y)$ whose saddle connections become arbitrarily short.

%%%%%%%%%%%%%%%%%%%
\subsection{Statement and discussion of the results}
 Let us call $V_{\pm 1}$ the open horocyclic neighborhood  of the lower cusp depicted in Figure \ref{horocycle}. Denote $\mathcal{Z}$  the union in $\Hy^2$ of all hyperbolic geodesics with endpoints in $\Z \cup \{\infty\}$, and their images under the Veech group $\Gamma$. Let us denote $\mbox{End}(\mathcal{Z})$ the set of pair of points at infinity of $\Hy^2$ which are endpoints of elements of $\mathcal{Z}$.
\begin{theorem}\label{disqueL(2,2)}
 For every $s \geq 2$, the minimum of $\mbox{KVol}$ over the Teichm\H{u}ller disk $\mathcal{T}_s$ is $(2s-1)\sqrt{\frac{143}{144}}$. It is achieved at the two points $(\pm \frac{9}{14}, \frac{\sqrt{143}}{14})$.
For $(x,y) \in V_{\pm 1}$, we have $\mbox{KVol}(X_s(x,y)) > 2s-1$, and $\mbox{KVol}(X_s(x,y))$ goes to infinity when $(x,y)$ tends to $(\pm 1, 0)$. Outside $V_{\pm 1}$, we have $\mbox{KVol}(X_s(x,y)) \leq 2s-1$, and $\mbox{KVol}(X_s(x,y))=2s-1$ if and only if $(x,y) \in \mathcal{Z}$. Furthermore $\mbox{KVol}(X_s(x,y))$ tends to $2s-1$ when $y\longrightarrow \infty$  while $(x,y)$ remains in $\mathcal{D}$.
\end{theorem}
\begin{figure}
\begin{center}
\begin{tikzpicture}[line cap=round,line join=round,>=triangle 45,x=2cm,y=2cm]
\clip(-2.9334693365393596,-0.5) rectangle (2.9198080888143774,1.5);
\draw [line width=0.8pt] (1,-1.776626971184012) -- (1,2.855627783952532);
\draw [line width=0.8pt] (-1,-1.776626971184012) -- (-1,2.855627783952532);
\draw [shift={(0,0)},line width=2pt,color=ffffff,fill=ffffff,fill opacity=0.1]  (0,0) --  plot[domain=2.085424903227732:3.141592653589793,variable=\t]({1*1*cos(\t r)+0*1*sin(\t r)},{0*1*cos(\t r)+1*1*sin(\t r)}) -- cycle ;
\draw [shift={(0,0)},line width=2pt,color=ffffff,fill=ffffff,fill opacity=0.1]  (0,0) --  plot[domain=2.085424903227732:3.141592653589793,variable=\t]({1*1*cos(\t r)+0*1*sin(\t r)},{0*1*cos(\t r)+1*1*sin(\t r)}) -- cycle ;
\draw [shift={(0,0)},line width=0.8pt]  plot[domain=0:3.141592653589793,variable=\t]({1*1*cos(\t r)+0*1*sin(\t r)},{0*1*cos(\t r)+1*1*sin(\t r)});
\draw [line width=1pt,dash pattern=on 2pt off 2pt,color=red, fill=red,fill opacity=0.10000000149011612] (-1,0.5) circle (1cm);
\draw [line width=1pt,dash pattern=on 2pt off 2pt,color=red, fill=red,fill opacity=0.10000000149011612] (1,0.5) circle (1cm);
\begin{scriptsize}
\draw [fill=ccqqqq] (1,0) circle (2pt);
%\draw[color=ccqqqq] (1.0577485167331366,0.1655622136302392) node {$A$};
%\draw[color=black] (0.9051206829559855,2.821286521352673) node {$f$};
\draw [fill=ccqqqq] (0,1) circle (2pt);
%\draw[color=ccqqqq] (0.08093038055937066,1.1881686999371535) node {$L_3$};
\draw [fill=ccqqqq] (-1,1) circle (2pt);
%\draw[color=ccqqqq] (-0.9416761057475405,1.1652745248705807) node {$D$};
\draw [fill=ccqqqq] (1,1) circle (2pt);
%\draw[color=ccqqqq] (1.0577485167331366,1.1652745248705807) node {$F$};
\draw [fill=ccqqqq] (-1,0) circle (2pt);
%\draw[color=ccqqqq] (-0.9416761057475405,0.15029943025252407) node {$S$};
\draw [fill=ccqqqq] (0,0) circle (2pt);
\end{scriptsize}
\end{tikzpicture}
\caption{The horocycle neighborhood $V_{\pm 1}$ of Theorem \ref{disqueL(2,2)}} \label{horocycle}
\end{center}
\end{figure}
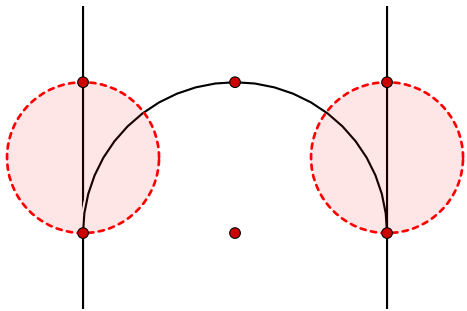
The idea of the proof of Theorem \ref{disqueL(2,2)} is that $\mbox{KVol} (X_s(x,y))$, has a simple expression as a function of the hyperbolic distance between 
$(x,y)$ and $\mathcal{Z}$.  The minimum is achieved by those points of $\mathcal{D}$ which are furthest away from $\mathcal{Z}$. Then the problem becomes a simple exercise in   hyperbolic geometry. 

Here is why the set $\mathcal{Z}$ enters the picture : given two directions $(p,q)$ and $(p',q')$, the saddle connections of respective directions $(p,q)$ and $(p',q')$ only have so much algebraic intersection to share between them all since they are lifted from closed curves on the torus. Thus, to realize KVol we must look for pairs of directions which have saddle connections which take up all the possible intersection. Those are the pairs of directions for which the invariant probability  measures supported on the saddle connections are least equidistributed. By Theorem 1.2 of \cite{Mc}, those pairs of directions correspond to those geodesics in $\mathcal{T}_s$ which are the least recurrent, that is, the geodesics which exit the fastest towards the cusps of $\mathcal{T}_s$. Those are precisely the geodesics in $\mathcal{Z}$. The precise meaning of exiting fastest towards the cusp is that they bounce at most twice on the lower boundary of $\mathcal{D}$-that is, the semi-circle between $1$ and $-1$- before going straight to a cusp. Another way of defining $\mathcal{Z}$ would be to say that the geodesic between $p/q$ and $p'/q'$ is in $\mathcal{Z}$ if and only if $(p/q,p'/q')$ is in the orbit, under the Veech group, of $(\infty, p''/q'')$, where the continued fraction expansion of $p''/q''$ has length $\leq 2$. 
A quantitative version of Theorem 1.2 of \cite{Mc} would be useful to generalize this work. In our proof, we use pedestrian arguments (Lemmata \ref{Irr'=1} and \ref{99_100}) instead.

It is proved in  \cite{JP} (see also \cite{HMS} and \cite{BG}) that $X_2(\exp (i \pi /3))$  minimizes the systolic volume in $\mathcal{H}(2)$. In the companion paper \cite{CKM}, we prove that $\inf \mbox{KVol} \leq 2 $ over  $\mathcal{H}(2)$. Since the systolic volume is a close relative of $\mbox{KVol}$, it is interesting to contrast the results of  \cite{JP}, \cite{BG} and  \cite{HMS} with ours. It is mildly surprising that the minimizers of KVol in $\mathcal{T}_s$ look, at first glance, pretty dull, while the interesting surfaces are (degenerate) local  maxima in $\mathcal{T}_s$. This prompts several questions. 
\begin{question}
Is KVol, as a function on $\mathcal{H}(2)$, differentiable at $St(3)$ ? Is it critical at $St(3)$ ?
\end{question}
\begin{question}
Is every square-tiled surface (ramified Riemannian covering of the square flat torus) in $\mathcal{H}(2)$, a local maximum of KVol in its own Teichm\H{u}ller disk ? Are they critical points of KVol  in $\mathcal{H}(2)$ ?
\end{question}
An easy consequence of Theorem \ref{disqueL(2,2)} is that for every $X \in \mathcal{T}$, the supremum in (\ref{defKVol}) is actually a maximum, and furthermore we identify the maximizers. We speculate this might be the case for every square-tiled surface ; since the union of the Teichm\H{u}ller disks of all square-tiled surfaces is dense in $\mathcal{H}(2)$, could it be true in the whole $\mathcal{H}(2)$ ? 
\begin{question}
Is it true that for every $X \in \mathcal{H}(2)$, the  supremum in (\ref{defKVol}) is actually a maximum ?
\end{question}
This would be in sharp contrast, both with the known behavior for flat tori, and with the expected behavior  for hyperbolic surfaces.\\ 
%%%%
\textbf{Acknowledgements} The first author acknowledges the support of a Profas B+ 2017 grant from Campus France, during the course of his Ph.D. studies.
%%%%%%%%%%%%%%%%%%%%%%%
%\section{$L(2,2)$ and its Teichm\H{u}ller disc}
\section{Preliminaries}\label{preliminaries}
\subsection{$St(2s-1)$}
%Following \cite{Schmithusen}, we call  $L(2,2)$  the 3-square translation surface of genus two, with one conical point,  depicted below, and w
We call  $e_1, \ldots e_s$, $e'_2, \ldots, e'_s$, (resp. $f_1, \ldots, f_s$, $f'_1, \ldots, f'_{s-1}$)  the horizontal (resp. vertical) closed curves of length $1$ in  $St(2s-1)$  obtained by gluing the endpoints of the sides of the squares (see Figure \ref{e1e2}). Note that $e_i$ and $e'_i$ (resp. $f_i$ and $f'_i$) are homotopic since they bound a cylinder. 

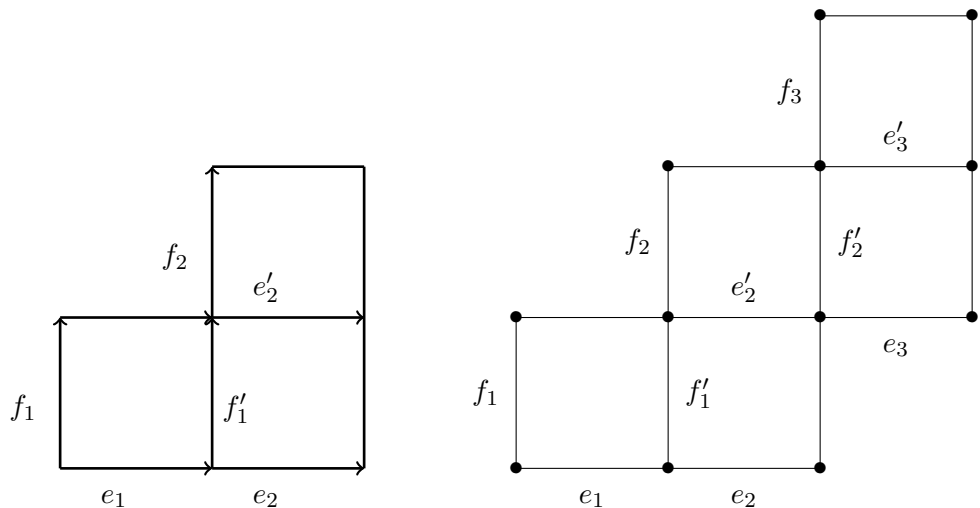
\begin{figure} [h!]
\begin{center}

\begin{tikzpicture}[scale=2]
\draw [->,line width=1pt,color=black] (0,0) -- (1,0);
\draw [->,line width=1pt,color=black] (1,0) -- (2,0);
\draw [->,line width=1pt,color=black] (0,0) -- (0,1);
\draw [->,line width=1pt,color=black] (1,1) -- (1,2);
\draw [->, line width=1pt,color=black] (1,0) -- (1,1);
%\draw [line width=1pt,color=black] (1,1) -- (2,2);
\draw [->,line width=1pt,color=black] (0,1) -- (1,1);
\draw [->, line width=1pt,color=black] (1,1) -- (2,1);
\draw [,line width=1pt,color=black] (2,0) -- (2,2);
\draw [line width=1pt,color=black] (1,2) -- (2,2);

\draw (0.2, -0.2) node[right] {$e_1$};
\draw (1.2, 1.2) node[right] {$e'_2$};
\draw (1.2, -0.2) node[right] {$e_2$};
\draw (-0.4, 0.4) node[right] {$f_1$};
\draw (1, 0.4) node[right] {$f'_1$};
\draw (0.6, 1.4) node[right] {$f_2$};

\draw (3,0) -- (5,0);
\draw (3,1) -- (6,1);
\draw (4,2) -- (6,2);
\draw (5,3) -- (6,3);
\draw (3,0) -- (3,1);
\draw (4,0) -- (4,2);
\draw (5,0) -- (5,3);
\draw (6,1) -- (6,3);
\draw (3,0) node {$\bullet$} ;
\draw (4,0) node {$\bullet$} ;
\draw (5,0) node {$\bullet$} ;
\draw (3,1) node {$\bullet$} ;
\draw (4,1) node {$\bullet$} ;
\draw (5,1) node {$\bullet$} ;
\draw (6,1) node {$\bullet$} ;
\draw (4,2) node {$\bullet$} ;
\draw (5,2) node {$\bullet$} ;
\draw (6,2) node {$\bullet$} ;
\draw (5,3) node {$\bullet$} ;
\draw (6,3) node {$\bullet$} ;
\draw (3.5,-0.2) node {$e_1 $};
\draw (4.5,1.2) node {$e'_2$};
\draw (4.5,-0.2) node {$e_2$};
\draw (5.5,2.2) node {$e'_3$};
\draw (2.8,0.5) node {$f_1$};
%\draw (5,0.5) node {$-$};
\draw (3.8,1.5) node {$f_2$};
\draw (5.2,1.5) node {$f'_2$};
\draw (5.5,0.8) node {$e_3$};
%\draw (5.5,3) node {$\backslash \backslash$};
\draw (4.8, 2.5) node {$f_3$};
\draw (4.2,0.5) node {$f'_1$};

%\draw (6.2, 2.5) node {x};
\end{tikzpicture}
\caption{ the closed curves $e_i, e'_i,  f_i, f'_i$ } \label{e1e2}
\end{center}
\end{figure}

Figure \ref{local pic} shows a local picture of $St(3)$ around the singular (conical) point $S$, with angles rescaled so the $6\pi$ fit into $2\pi$.
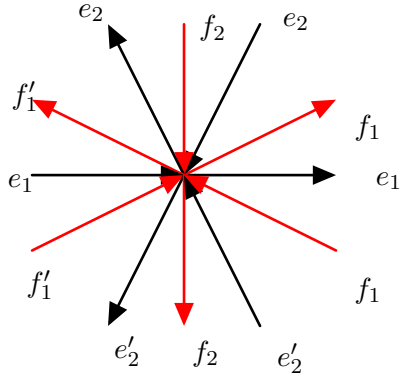
\begin{figure}[h!]
\begin{center}
\begin{tikzpicture}[line cap=round,line join=round,>=triangle 45,x=1cm,y=1cm]
\clip(-3.36621647446458,-3.1902) rectangle (3.786216474464581,2.4702);
\draw [->,line width=1pt,color=black] (0,0) -- (2,0);
\draw [->,line width=1pt,color=red] (0,0) -- (2,1);
\draw [->,line width=1pt,color=black] (1,2) -- (0,0);
\draw [->,line width=1pt,color=red] (0,2) -- (0,0);
\draw [->,line width=1pt,color=black] (0,0) -- (-1,2);
\draw [->,line width=1pt,color=red] (0,0) -- (-2,1);
\draw [->,line width=1pt,color=black] (-2,0) -- (0,0);
\draw [->,line width=1pt,color=red] (-2,-1) -- (0,0);
\draw [->,line width=1pt,color=black] (0,0) -- (-1,-2);
\draw [->,line width=1pt,color=red] (0,0) -- (0,-2);
\draw [->,line width=1pt,color=black] (1,-2) -- (0,0);
\draw [->,line width=1pt,color=red] (2,-1) -- (0,0);
\draw (2.3874355848434936,0.17619934102141693) node[anchor=north west] {$e_1$};
\draw (2.107679406919276,0.9408662273476115) node[anchor=north west] {$f_1$};
\draw (1.165833607907744,2.2836958813838555) node[anchor=north west] {$e_2$};
\draw (0.06545930807248848,2.2557202635914337) node[anchor=north west] {$f_2$};
\draw (-1.5291509060955513,2.3955983525535425) node[anchor=north west] {$e_2$};
\draw (-2.4057202635914328,1.3791509060955522) node[anchor=north west] {$f'_1$};
\draw (-2.452346293245469,0.13889851729818795) node[anchor=north west] {$e_1$};
\draw (-2.2098909390444805,-1.1200042833607908) node[anchor=north west] {$f'_1$};
\draw (-1.0535654036243816,-2) node[anchor=north west] {$e'_2$};
\draw (-0.027792751235584022,-2.061850082372323) node[anchor=north west] {$f_2$};
\draw (1.0912319604612861,-2.099150906095552) node[anchor=north west] {$e'_2$};
\draw (2.107679406919276,-1.2225815485996707) node[anchor=north west] {$f_1$};
\end{tikzpicture}
\caption{Local picture around the conical point}\label{local pic}
\end{center}
\end{figure}

The local picture at $S$ in $St(2s-1)$ may be obtained by induction : start from the local picture in $St(2s-3)$, split $f_{s-1}$ into $f_{s-1}$ and $f'_{s-1}$, and insert, between  $f_{s-1}$ and $f'_{s-1}$, $f_s$, $e_s$, and $e'_s$, as shown in Figure \ref{insert}.
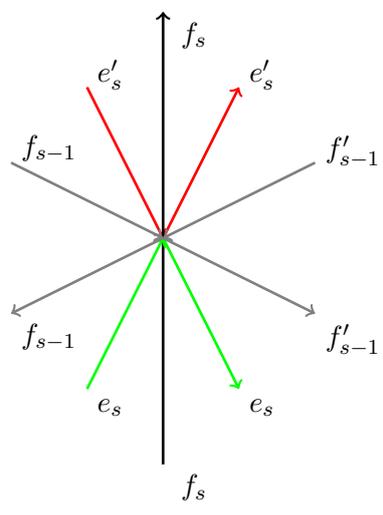
\begin{figure}[!h]
\begin{center}
\begin{tikzpicture}
\draw [->,line width=1pt,color=red] (-1,2) -- (0,0);
\draw [->,line width=1pt,color=red] (0,0) -- (1,2);
\draw (-1,2.5) node[anchor=north west] {$e'_s$};
\draw (1,2.5) node[anchor=north west] {$e'_s$};
\draw (-1,-2) node[anchor=north west] {$e_s$};
\draw (1,-2) node[anchor=north west] {$e_s$};
\draw [->,line width=1pt,color=black] (0,-3) -- (0,3);
\draw (0.1,3) node[anchor=north west] {$f_s$};
\draw (0.1,-3) node[anchor=north west] {$f_s$};
\draw [->,line width=1pt,color=green] (-1,-2) -- (0,0);
\draw [->,line width=1pt,color=green] (0,0) -- (1,-2);
\draw [->,line width=1pt, color=gray] (-2,1) -- (0,0);
\draw [->,line width=1pt, color=gray] (0,0) -- (-2,-1);
\draw (-2,1.5) node[anchor=north west] {$f_{s-1}$};
\draw (-2,-1) node[anchor=north west] {$f_{s-1}$};
\draw [->,line width=1pt, color=gray] (2,1) -- (0,0);
\draw [->,line width=1pt, color=gray] (0,0) -- (2,-1);
\draw (2,1.5) node[anchor=north west] {$f'_{s-1}$};
\draw (2,-1) node[anchor=north west] {$f'_{s-1}$};
\end{tikzpicture}
\caption{Insertion of $f_s$, $e_s$, $e'_s$}\label{insert}
\end{center}
\end{figure}

Since $e_1, e_2, f_1,  f_2$ do not meet anywhere but at $S$, the local picture yields the algebraic intersections between any two of $e_1, e_2, f_1,  f_2$, summed up in the following matrix: 
\[
\begin{array}{ccccc}
\mbox{Int} & e_1& f_1 & e_2 & f_2 \\
e_1            & 0     &   1   &  0   &  -1    \\
f_1            & -1    &   0   &  0   &  0  \\
e_2             & 0     &   0  & 0    &  1   \\
f_2             & 1    &  0    & -1   &  0 
\end{array}
\]
Let us call $\alpha_1, \ldots, \alpha_s$  (resp. $\beta_1, \ldots, \beta_s$) the closed, non-singular horizontal (resp. vertical)  geodesics depicted in Figure \ref{nonsing}.
\begin{figure} [h!]\label{alpha1beta1}
	\begin{center}
		
		\begin{tikzpicture}[scale=2]
		
		\draw (3,0) -- (5,0);
		\draw (3,1) -- (6,1);
		\draw (4,2) -- (6,2);
		\draw (5,3) -- (6,3);
		\draw (3,0) -- (3,1);
		\draw (4,0) -- (4,2);
		\draw (5,0) -- (5,3);
		\draw (6,1) -- (6,3);
		\draw (3,0) node {$\bullet$} ;
		\draw (4,0) node {$\bullet$} ;
		\draw (5,0) node {$\bullet$} ;
		\draw (3,1) node {$\bullet$} ;
		\draw (4,1) node {$\bullet$} ;
		\draw (5,1) node {$\bullet$} ;
		\draw (6,1) node {$\bullet$} ;
		\draw (4,2) node {$\bullet$} ;
		\draw (5,2) node {$\bullet$} ;
		\draw (6,2) node {$\bullet$} ;
		\draw (5,3) node {$\bullet$} ;
		\draw (6,3) node {$\bullet$} ;
		
		%\draw (5.5,3) node {$\backslash \backslash$};
		%\draw (6.2, 2.5) node {x};
		\draw [->, line width = 1pt, color=red](3,0.5) -- (5,0.5);
		\draw (3,0.7) node[anchor=north west, color=red] {$\alpha_1$};
		\draw [->, line width = 1pt, color=red](4,1.5) -- (6,1.5);
		\draw (4,1.7) node[anchor=north west, color=red] {$\alpha_2$};
		\draw [->, line width = 1pt, color=red](5,2.5) -- (6,2.5);
		\draw (5,2.7) node[anchor=north west, color=red] {$\alpha_3$};
		\draw [->,  line width = 1pt, color=green](3.5,0) -- (3.5,1);
		\draw (3.2,0) node[anchor=north west, color=green] {$\beta_1$};
		\draw [->,  line width = 1pt, color=green](4.5,0) -- (4.5,2);
		\draw (4.2,0) node[anchor=north west, color=green] {$\beta_2$};
		\draw [->,  line width = 1pt, color=green](5.5,1) -- (5.5,3);
		\draw (5.2,1) node[anchor=north west, color=green] {$\beta_3$};
		
		\end{tikzpicture}
		\caption{The non-singular geodesics $\alpha_i$, $\beta_i$, $i=1,2,3$}  \label{nonsing}
	\end{center}
\end{figure}
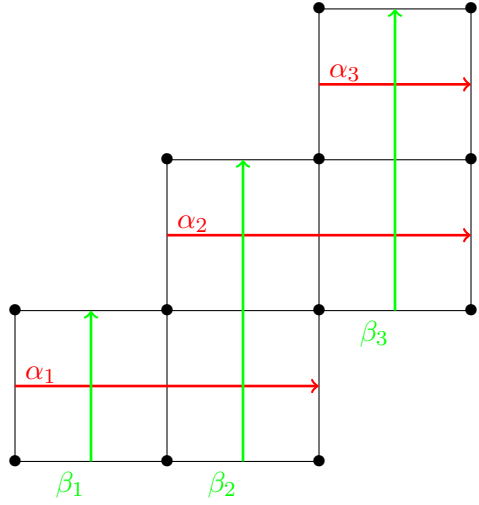
\begin{lemma}
	For $ i,j=1, \ldots, s$ we have :
	$$ \mbox{Int}(e_{i},f_{j})= 
	\left\{
	\begin{array}{ccc}
	(-1)^{j-i} & \mbox{if}  & j\geq  i \\
	 0 & \mbox{if}  & j< i
	\end{array}
	\right.
	$$
	and $$ \mbox{Int}(e_{i},e_{j})=\mbox{Int}(f_{i},f_{j})=0$$
	
\end{lemma}
\begin{proof}
Observe that the closed, non-singular vertical geodesic $\beta_1$ is homotopic to $f_{1} $ and for each $j= 2,\ldots,s $, $\beta_j$ are homotopic to $f_{j}+f_{j-1}$. We observe that
$$ \mbox{Int}(e_{i},\beta_{j})= 
\left\{
\begin{array}{ccc}
1 & \mbox{if}  & i=j \\
0 & \mbox{if}  & i\neq j
\end{array}
\right.
$$
Then we have :\\
For $j=1$, $\mbox{Int}(e_{1},f_{1})=1$ and for $i>1$,  $\mbox{Int}(e_{i},f_{1})=0$.\\
For $ j\geq 2$ we have three cases:\\
1) if $ i>j$ we have $\mbox{Int}(e_{i},\beta_{j})=0$ which implies $ \mbox{Int}(e_{i},f_{j})=-\mbox{Int}(e_{i},f_{j-1})$, by induction we deduce $$ \mbox{Int}(e_{i},f_{j})=(-1)^{i-j}\mbox{Int}(e_{i},f_{1})=0$$ 
2) if $i=j$ we have $$ \mbox{Int}(e_{i},f_{i})+\mbox{Int}(e_{i},f_{i-1})=1$$
but by the first case $ \mbox{Int}(e_{i},f_{i-1})=0$, so 
$$ \mbox{Int}(e_{i},f_{i})=1$$\\
3) if $ i<j$  we have $$ \mbox{Int}(e_{i},f_{j})=-\mbox{Int}(e_{i},f_{j-1})$$
By induction and using the second case, we obtain 
\begin{align*}
\mbox{Int}(e_{i},f_{j})&=(-1)^{j-i}\mbox{Int}(e_{i},f_{i})\\
&=(-1)^{j-i}.
\end{align*}
Observe that the closed, non-singular horizontal geodesic $\alpha_s,$ is homotopic to $e_{s} $ and for each $j=1,\ldots,s-1 $ , $\alpha_j$ are homotopic to $e_{j}+e_{j+1}$. Then for  $i=1,\ldots,s$, $\mbox{Int}(e_{i},e_{s})=0 $
 and for $i=1,\ldots,s$, $j=1,\ldots,s-1 $
 $\mbox{Int}(e_{i},\alpha_{j})=0 $. Finally, we obtain 
\begin{align*}
\mbox{Int}(e_{i},e_{j})&=-\mbox{Int}(e_{i},e_{j+1})\\ 
&= (-1)^{s-j}\mbox{Int}(e_{i},e_{s})=0.
\end{align*}
In an analogous way we have  $\mbox{Int}(f_{i},f_{j})=0$.
\end{proof}
For $s >2$, the intersection matrix is, in the basis $e_1, f_1, \ldots, e_s, f_s$  (in this example $s=4$):
\[
\left(
\begin{array}{cc|cc|cc|cc}

    0     &   1   &  0   &  -1 & 0 & 1 & 0 &-1  \\
    -1    &   0   &  0   &  0  & 0 & 0 &  0 & 0 \\
    \hline
     0    &   0   & 0    &  1  & 0 & -1 & 0 & 1  \\
    1    &   0    & -1   &  0  & 0 & 0 & 0 & 0 \\
    \hline
     0    &  0    & 0    &  0  & 0 & 1 & 0 & -1 \\
      -1  &  0    & 1    &  0  & -1& 0 & 0 & 0 \\
      \hline
    0    &  0    & 0    &  0  & 0 & 0 & 0 & 1 \\
    1    &  0    & -1  &  0  & 1 & 0 & -1 & 0 \\
\end{array}
\right).
\]
Since the determinant of the intersection matrix is not zero, we get for free the fact that $e_i, f_i$, $i=1,\ldots, s$ form a basis of $H_1 (St(2s-1), \R)$. From now on we always refer to homology classes in $H_1 (St(2s-1), \R)$ by their coordinates $\epsilon_1, \dots, \epsilon_s, \phi_1, \ldots, \phi_s$ in the basis $e_i, f_i$.

Observe that the homology classes of $\alpha_1, \ldots, \alpha_{s-1}, \alpha_s$ are $\left[ e_1\right]+ \left[ e_2 \right], \ldots, \left[ e_{s-1}\right]+ \left[ e_s \right], \left[ e_s \right]$, while the homology classes of 
$\beta_1, \beta_2, \ldots,  \beta_s$ are $ \left[ f_1 \right], \left[ f_1\right]+ \left[ f_2 \right], \ldots, \left[ f_{s-1}\right]+ \left[ f_s \right]$.

Thus, using the intersection matrix, we see that, for any homology class $h$ with coordinates $\epsilon_1, \dots, \epsilon_s, \phi_1, \ldots, \phi_s$, we have 
\begin{align*}
\epsilon_i & = \mbox{Int}(h, \left[ \beta_i \right])    \\
\phi_i & = -\mbox{Int}(h, \left[ \alpha_i \right]).   
\end{align*}

\subsection{A short excursion into the Teichm\H{u}ller space of flat tori}\label{A short excursion into the Teichmüller space of flat tori}
Recall that the Teichm\H{u}ller space of flat tori is the hyperbolic plane ; one way to see that is to fix a homology class of simple closed curves, then   every flat torus is biholomorphic to 
\[
T(x,y) = \R^2 /  \Z(1,0)\oplus \Z(x,y)
\]
where $(x,y) \in \R \times \R_+^*$ and the biholomorphism sends a simple closed curve in the homology class $h$ to the image in $T(x,y)$ of $(0,1)$. 

%Homology classes in $H_1(T(x,y), \Z)$ will be written in coordinates with respect to the basis $(1,0), (x,y)$, so $(p,q)$ is the homology class of the simple closed curve which is the image in $T(x,y)$ of 
%$p(0,1)+q(x,y)$. We say a closed geodesic on $T(x,y)$ has direction $(p,q)$ if its homology class is $(p,q)$. 

Now take $p,q,p',q' \in \Z$ such that $p \wedge q =  p' \wedge q' = 1$, set $r = p/q, r'=p'/q'$, $r$ (resp. $r'$) being understood as the point at infinity if $q=0$ (resp. $q'=0$). Let $\gamma_{r,r'}$ be the hyperbolic geodesic with endpoints $r$ and $r'$. Recall that $\gamma_{r,r'}$ is the locus, in the Teichm\H{u}ller space, of the flat tori in which the closed geodesics with respective directions $(p,q)$ and $(p',q')$ are orthogonal. 

Denote by $d_{r,r'} \  : \  \Hy^2 \longrightarrow \R$ the distance function to the geodesic $\gamma_{r,r'}$. Call $E_{r,r'}(d)$ the $d$-level set of $d_{r,r'}$, for any $d \geq 0$. Pick $(x,y)$ in $\Hy^2$, let 
$d= d_{r,r'}(x,y)$, and let   $\theta_{r,r'}(x,y)$ be the angle between $E_{r,r'}(d)$ and $\gamma_{r,r'}$ (see Figure \ref{banana}).

\begin{figure}
\begin{center}
\begin{tikzpicture}[line cap=round,line join=round,>=triangle 45,x=1cm,y=1cm, scale=1.8]
\clip(-3,0) rectangle (2.6,4);
\draw [shift={(-2,0)},line width=1pt,fill=red,fill opacity=0.10000000149011612] (0,0) -- (53.13010235415601:0.23134810951760104) arc (53.13010235415601:90:0.23134810951760104) -- cycle;
\draw [shift={(-2,0)},line width=1pt,fill=red,fill opacity=0.10000000149011612] (0,0) -- (90:0.23134810951760104) arc (90:126.86989764584399:0.23134810951760104) -- cycle;
\draw [shift={(0,0)},line width=1pt,color=black]  plot[domain=0:3.141592653589793,variable=\t]({1*2*cos(\t r)+0*2*sin(\t r)},{0*2*cos(\t r)+1*2*sin(\t r)});
\draw [shift={(0,1.5)},line width=1pt,color=red]  plot[domain=-0.6435011087932843:3.7850937623830774,variable=\t]({1*2.5*cos(\t r)+0*2.5*sin(\t r)},{0*2.5*cos(\t r)+1*2.5*sin(\t r)});
\draw [shift={(0,-1.5)},line width=1pt,color=red]  plot[domain=0.6435011087932844:2.498091544796509,variable=\t]({1*2.5*cos(\t r)+0*2.5*sin(\t r)},{0*2.5*cos(\t r)+1*2.5*sin(\t r)});
\draw [->,line width=1pt,color=red] (-2,0) -- (-1.4084909621903532,0.7886787170795292);
\draw [->,line width=1pt,color=black] (-2,0) -- (-2,1);
\draw [->,line width=1pt,color=red] (-2,0) -- (-2.6164052672750966,0.8218736897001288);
\draw [line width=2pt,color=wrwrwr,domain=-3.0190928292046935:2.895707170795306] plot(\x,{(-0-0*\x)/4});
\draw (-2.001161147327249,0) node[anchor=north west] {$r=\frac{p}{q}$};
\draw (1.9857379400260755,0) node[anchor=north west] {$r'=\frac{p'}{q'}$};
\draw (-0.4973984354628422,2.5) node[anchor=north west] {$\gamma_{r,r'}$};
\draw (-0.3971475880052151,1.0268959582790067) node[anchor=north west] {$E_{r,r'} (d)$};
\draw (-0.28918513689700126,3.9727285528031264) node[anchor=north west] {$E_{r,r'} (d)$};
\begin{scriptsize}
\draw[color=black] (-1.8,0.4) node {$\theta$};
\draw[color=black] (-2.1,0.4) node {$\theta$};
\end{scriptsize}
\end{tikzpicture}
\caption{A banana neighborhood} \label{banana}
\end{center}
\end{figure}
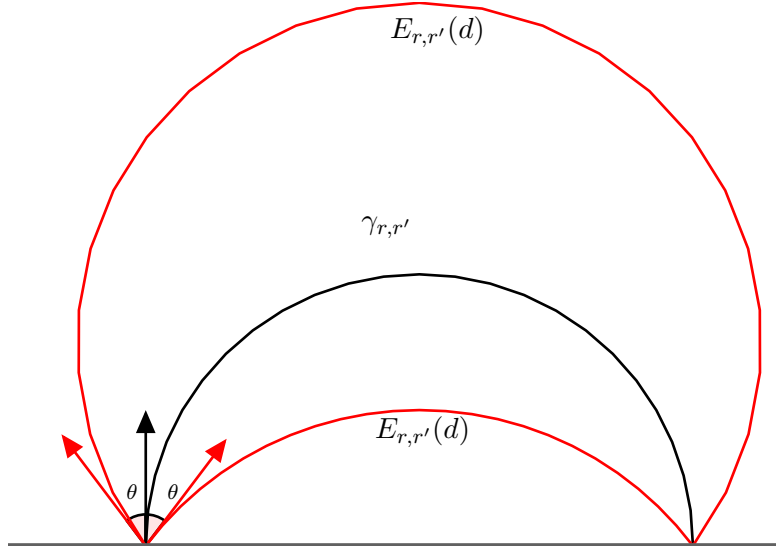
Denote $l_{r}(x,y)$ the length, in the flat torus $T(x,y)$, of the closed geodesics with direction $(p,q)$. 
Set 
\[
K_{r,r'}(x,y) = \mbox{Vol}(T(x,y)) \frac{pq'-p'q}{l_{r}(x,y)l_{r'}(x,y)}.
\]
This definition is tailored so 
\[
\mbox{KVol}(T(x,y)) = \sup_{r,r'} K_{r,r'}(x,y) .
\]
\begin{lemma}\label{sl2}
For any $g \in SL_2 (\Z)$, we have
\[
K_{r,r'}(g(x,y))= K_{g(r),g(r')}(x,y)).
\]
\end{lemma}
\begin{proof}
First, $pq'-p'q$ and $\mbox{Vol}(T(x,y))$ are invariant under the orientation- and volume-preserving diffeomorphism $g$. Second, $l_{r}(g(x,y))=l_{g(r)}(x,y)$.
\end{proof}
\begin{lemma}
For any $(x,y) \in \Hy^2$ and $p,q,p',q' \in \Z$ such that $p \wedge q =  p' \wedge q' = 1$, we have, setting $r = p/q, r'=p'/q'$, 
\[
K_{r,r'}(x,y)=\cos \theta_{r,r'}(x,y).
\]
\end{lemma}
\begin{proof}
By Lemma \ref{sl2}, acting by some element of $SL_2 (\Z)$ if we need to, we may assume that $p'=1, q'=0$, that is, $\gamma_{r,r'}$ is vertical. Then $l_{r'}(x,y)=1$, $l_{r}(x,y)=\sqrt{(p+qx)^2 +(qy)^2}$, $pq'-p'q=q$, and $\mbox{Vol}(T(x,y)) =y$, so 
\[
K_{r,r'}(g(x,y))=\frac{qy}{\sqrt{(p+qx)^2 +(qy)^2}}.
\]
Take radial coordinates $(\rho, \alpha)$ with origin at $-r$, so $x+p/q = \rho \cos \alpha$ and $y= \rho \sin \alpha$. Then 
\[
K_{r,r'}(g(x,y))=\frac{\rho \sin \alpha}{\rho} = \sin \alpha.
\]
Now observe on Figure \ref{alphatheta}  that $\alpha=\pi /2 - \theta_{r,r'} (x,y)$. This concludes the proof. 
\end{proof}
\begin{figure}
\begin{center}
\begin{tikzpicture}[line cap=round,line join=round,>=triangle 45,x=1cm,y=1cm]
\clip(-4,-1) rectangle (5,5);
\draw [shift={(0,0)},line width=1pt,fill=red,fill opacity=0.10000000149011612] (0,0) -- (63.43494882292201:0.6442920469361146) arc (63.43494882292201:90:0.6442920469361146) -- cycle;
\draw [shift={(0,0)},line width=1pt,fill=black,fill opacity=0.10000000149011612] (0,0) -- (0:0.6442920469361146) arc (0:63.43494882292201:0.6442920469361146) -- cycle;
\draw [line width=1pt,color=black,domain=-8.2362:8.2362] plot(\x,{(-0-0*\x)/7});
\draw [line width=1pt,color=black] (0,0) -- (0,9.618354106910038);
\draw [line width=1pt,color=red,domain=0:8.236199999999998] plot(\x,{(-0--6*\x)/3});
\draw [line width=1pt,color=red,domain=-8.236199999999998:0] plot(\x,{(-0--6*\x)/-3});
\draw (-1.8,4.5) node[anchor=north west] {$(x,y)$};
\draw (-0.5,0) node[anchor=north west] {$-\frac{p}{q}$};
\draw (0,4) node[anchor=north west] {$\gamma_{r,r'}$};
\draw (0.03221460234680573,2.3808067796610173) node[anchor=north west] {$\theta_{r,r'} (x,y)$};
\draw (0.6,0.8) node[anchor=north west] {$\alpha$};
\draw (0.010738200782268576,5.387502998696219) node[anchor=north west] {$\gamma_{r,r'}$};
\draw (-3.5,3.5) node[anchor=north west] {$E_{r,r'} (x,y)$};
\draw (1.9865671447196867,4.313682920469361) node[anchor=north west] {$E_{r,r'} (x,y)$};
\begin{scriptsize}
\draw [fill=rvwvcq] (-2,4) circle (2.5pt);

\end{scriptsize}
\end{tikzpicture}
\caption{$\alpha=\pi /2 - \theta_{r,r'} (x,y)$}\label{alphatheta}
\end{center}
\end{figure}
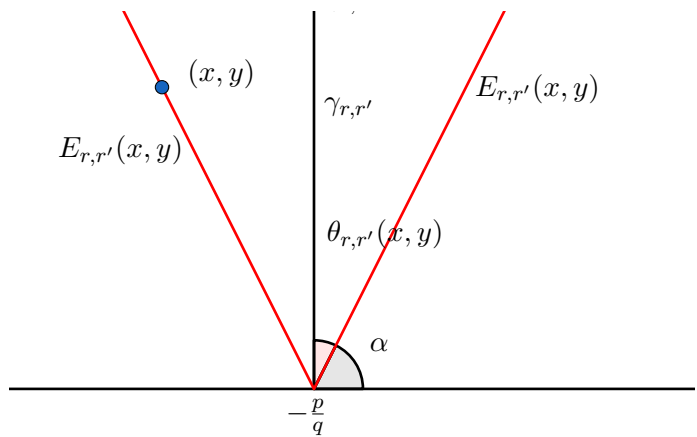

\subsection{Saddle connections and intersection}\label{subsecSaddle connections and intersection}
Recall that a saddle connection is a closed geodesic which contains a cone point. Saddle connections on surfaces in the Teichm\H{u}ller disk of $St(2s-1)$ have a special property : 
\begin{lemma}\label{1to1}
If $X \in \mathcal{T}_s$ and $\gamma$ is a saddle connection on $X$, then the projection $\Pi : X \longrightarrow \T^2$, restricted to $\gamma$,  is 1-to-1.
\end{lemma}
\begin{proof}
Parametrize $\gamma$ at unit speed by $\left[0,l(\gamma) \right]$, so that $\gamma(0)=\gamma(l(\gamma ))=S$, $S$ being the cone point of $X$. Assume $\Pi(\gamma(t_1))=\Pi(\gamma(t_2))$ for $t_1, t_2 \in \left[0,l(\gamma) \right]$, with $t_1 \neq t_2$. Then $\Pi ( \gamma)( \left[ t_1,t_2 \right])$ is a closed subarc of $\Pi ( \gamma)(\left[0,l(\gamma) \right])$, parametrized at unit speed, so it must contain all of $\Pi ( \gamma)(\left[0,l(\gamma) \right])$. Therefore there exists $t_3 \in  \left[ t_1,t_2 \right]$ such that $\Pi (\gamma) (t_3)=(0,0)= \Pi (S)$. Now recall that $S$ is the only pre-image of $(0,0)$ (this is where surfaces in 
$\mathcal{T}_s$ are special), and the only pre-images of $S$ by the chosen parametrization of $\gamma$ are $0$ and $l(\gamma)$, so $t_3=0$ or  $t_3=l(\gamma)$. Then $\Pi (\gamma) (t_1)=\Pi (\gamma) (t_2)=(0,0)$, whence $\gamma (t_1)= \gamma (t_2)=S$, which proves that $t_1=0, t_2=l(\gamma)$, and the lemma. 
\end{proof}
%An element $X(x,y)$ of $\mathcal{T}$  is a $3$-fold cover of the flat torus $\R^2 / (1,0)\Z \oplus (x,y)\Z$, with $|x| \leq 1$ and $x^2+y^2 \geq 1$, ramified over a single point. We say a geodesic arc on $X$ has direction $(p,q)$, for coprime $p,q \in \Z$, if it projects to a straight segment in the flat torus,  parallel to the vector $p(1,0) +q (x,y)$.  

The conical point is the only pre-image of the ramification point, unlike the case of the surface $L(n,n)$
(see \cite{CKM}), so by Lemma \ref{1to1} a closed geodesic of direction $(p,q)$ projects 1-to-1 to a closed geodesic in the flat torus. Therefore a saddle connection of direction $(p,q)$ is exactly $\sqrt{(p+qx)^2+(qy)^2}$-long. For future reference we state this fact as a lemma.
\begin{lemma}\label{longueur}
Given $(x,y) \in \Hy^2$ and $(p,q) \in \Z^2$ with $p \wedge q =  1$, for any saddle connection $\alpha$ with direction $(p,q)$ in the translation surface $X_s(x,y)$, we have, using the notation of 
Subsection \ref{A short excursion into the Teichmüller space of flat tori}, $l(\alpha)= l_{r}(x,y)$.
\end{lemma}

Here is why Subsection \ref{A short excursion into the Teichmüller space of flat tori} is relevant to the determination of KVol in $\mathcal{T}_s$. For $(x,y) \in \Hy^2$, and for $p,q,p',q' \in \Z$ such that $p \wedge q =  p' \wedge q' = 1$, setting $r = p/q, r'=p'/q'$, and calling $\alpha_i, i=1, \ldots,2s-1$ (resp. $\beta_i, i=1, \ldots,2s-1$) the saddle connections with direction $(p,q)$ (resp. $(p',q')$) in the translation surface $X_s(x,y)$, we have $l(\alpha_i)= l_{r}(x,y)$ and $l(\beta_i)= l_{r'}(x,y)$ for $i=1,\ldots,2s-1$, and $\mbox{Vol} (X_s(x,y)) = (2s-1) \mbox{Vol} (T(x,y)) $, since $X_s(x,y)$ is a (2s-1)-fold ramified Riemannian cover of $T(x,y)$.  Hence 
\begin{align}
\mbox{KVol}(X_s(x,y)) &\geq \sup_{r \neq r'} \sup_{i,j} \mbox{Vol} (X_s(x,y))\frac{\mbox{Int}(\alpha_i, \beta_j)}{l(\alpha_i)l(\beta_j)}\\
&= (2s-1) \sup_{r \neq r'} \sup_{i,j}  \frac{\mbox{Int}(\alpha_i, \beta_j)}{pq'-p'q} K_{r,r'}(x,y). \label{etoile}
\end{align}
The reason why there is a $\geq$ instead of $=$ in the inequality above is that there could be saddle connections with the same direction, i.e. $r=r'$, and non-zero intersection, in which case $pq'-p'q=0$ and the formula above does not apply ; more on that in subsection \ref{Irr'}. For the time being, we set
\[
I_{r,r'}(x,y)= \sup_{i,j}  \frac{\mbox{Int}(\alpha_i, \beta_j)}{pq'-p'q}
\]
and study $I_{r,r'}(x,y)$ as a function of $r$ and $r'$. Note that actually $I_{r,r'}(x,y)$ does not depend on $(x,y)$,  so from now on we shall denote it $I_{r,r'}$. 
\subsubsection{Intersections and $I_{r,r'}$}\label{Irr'}
Recall from \cite{HL} that the action on $\Z^2$ of the Veech group $\Gamma$ of $St(2s-1)$  has two orbits, that of $(1,0)$, which consists of vectors whose coordinates are not equal modulo $2$, and that of $(1,1)$, which consists of vectors whose coordinates are equal modulo $2$.
Therefore, since $\Gamma$ acts by orientation-preserving diffeomorphisms,  for $(p,q) \in \Z^2$, with $ p \neq q \mod 2$, the saddle connections in $St(2s-1)$ with direction $(p,q)$ are the images under some diffeomorphism of $e_1,e_i, e'_i, i=2,\ldots s$, so they have zero mutual intersection. 

On the other hand, for $p,q \in \Z$ both odd,  the saddle connections in $St(2s-1)$ with direction $(p,q)$ are the images under some diffeomorphism of $ g_i, g'_{i}, g_s, i=1,\ldots s-1$ depicted in Figure \ref{g1g2g3}.

\begin{figure} [h!]
	\begin{center}
		
		\begin{tikzpicture}[scale=2]
		
		\draw (3,0) -- (5,0);
		\draw (3,1) -- (6,1);
		\draw (4,2) -- (6,2);
		\draw (5,3) -- (6,3);
		\draw (3,0) -- (3,1);
		\draw (4,0) -- (4,2);
		\draw (5,0) -- (5,3);
		\draw (6,1) -- (6,3);
		\draw (3,0) node {$\bullet$} ;
		\draw (4,0) node {$\bullet$} ;
		\draw (5,0) node {$\bullet$} ;
		\draw (3,1) node {$\bullet$} ;
		\draw (4,1) node {$\bullet$} ;
		\draw (5,1) node {$\bullet$} ;
		\draw (6,1) node {$\bullet$} ;
		\draw (4,2) node {$\bullet$} ;
		\draw (5,2) node {$\bullet$} ;
		\draw (6,2) node {$\bullet$} ;
		\draw (5,3) node {$\bullet$} ;
		\draw (6,3) node {$\bullet$} ;
		\draw (3.5,-0.2) node {$e_1 $};
		\draw (4.5,-0.2) node {$e_2$};
		\draw (2.8,0.5) node {$f_1$};
		%\draw (5,0.5) node {$-$};
		\draw (3.8,1.5) node {$f_2$};
		
		\draw (5.5,0.8) node {$e_3$};
		%\draw (5.5,3) node {$\backslash \backslash$};
		\draw (4.8, 2.5) node {$f_3$};
		%\draw (6.2, 2.5) node {x};
		\draw [->, line width= 1pt, color=red] (3,0) -- (4,1);
		\draw (3.5,.5) node[anchor=north west, color=red] {$g_1$};
		\draw [->, line width= 1pt, color=red] (4,0) -- (5,1);
		\draw (4.5,.5) node[anchor=north west, color=red] {$g'_1$};
		\draw [->, line width= 1pt, color=red] (4,1) -- (5,2);
		\draw (4.5,1.5) node[anchor=north west, color=red] {$g_2$};
		\draw [->, line width= 1pt, color=red] (5,1) -- (6,2);
		\draw (5.5,1.5) node[anchor=north west, color=red] {$g'_2$};
		\draw [->, line width= 1pt, color=red] (5,2) -- (6,3);
		\draw (5.5,2.5) node[anchor=north west, color=red] {$g_3$};
		\end{tikzpicture}
		\caption{the saddle connections of direction $(1,1)$} \label{g1g2g3}
	\end{center}
\end{figure}
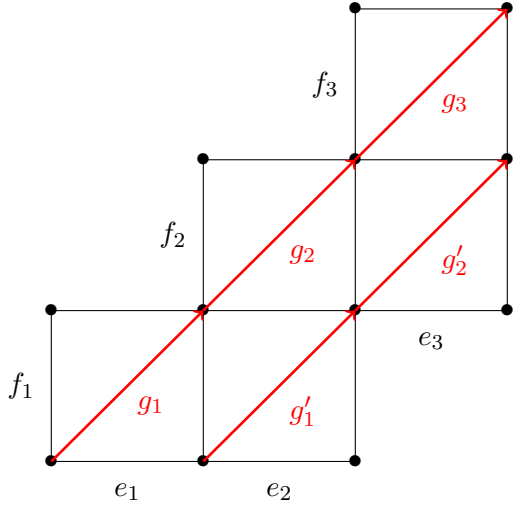
The homology classes of $g_1, g_i, g'_{i-1}, i=2,\ldots s$ are, in the basis $\left\lbrace e_1,\ldots,e_s,f_1,\ldots,f_s\right\rbrace $, $ g_1=e_1+f_1$, and $g_i=e_i + f_i$, $ g'_{i-1} = e_i +f_{i-1}$, so using the intersection matrix we see that 

$$ \mbox{Int}(g_{i} ,g_{j} )=\mbox{Int}(g'_{j-1} , g'_{i-1} )= 
\left\{
\begin{array}{ccc}
(-1)^{j-i} & \mbox{if}  & j> i \\
-(-1)^{i-j} & \mbox{if}  & j<i\\
\end{array}
\right.
$$
and for any $ i=1,\ldots s$,  $j=2,\ldots s$, $ \mbox{Int}(g_{i} ,g'_{j-1})=-(-1)^{j-i}$.

Since the  game is to find curves which intersect a lot, our next task is to maximize $I_{r,r'}$ as a function of $r,r'$.
The next two lemmas are the reason why the geodesics in $\mathcal{Z}$ are special. 
\begin{lemma}\label{Irr'=1}
For any $(r,r') \in \mbox{End}(\mathcal{Z})$, we have $I_{r,r'}=1$.
\end{lemma}
\begin{proof}
Let $(p,q), (p',q') \in \Z^2$ be such that $p \wedge q =  p' \wedge q' = 1$, and $p/q=r, p'/q'=r'$. We want to prove that for any surface $X_s(x,y) \in \mathcal{T}_s$, there exist saddle connections $\gamma, \gamma'$, with respective directions $(p,q)$ and $(p',q')$, such that $\mbox{Int}(\gamma, \gamma') = pq'-p'q$. 
By the definition of $\mbox{End}(\mathcal{Z})$, there exists $V \in \Gamma$ such that $(V (r),V(r')) \in \Z^2 \cup \{\infty\}$, where $V(r)$ is understood as the projective action of $\Gamma$ on the projective line. Since $\Gamma$ has two orbits in $\Q$, that of $\infty$ (or $0$), which corresponds to $p \neq q \mod 2$, and that of $1$, which corresponds to $p=q=1 \mod 2$, we may assume $V(r)=\infty$ (or $0$) or $V(r)=1$.
 
\textbf{First case} :
$V(r)=0$, that is, $V(p,q)=(0,1)$ (where $V(p,q)$ is understood as the linear action of $\Gamma$ on $\R^2$). Since $V(r') \in \Z$, there exists $n \in \Z$ such that $V(p',q') = (-n,1)$. Then (recall that $R^3=-R$), $R^3 \circ V (p',q')=(1, n)$ and $R^3 \circ V (p,q)=(1, 0)$. Since 
$\Gamma$ is a subgroup of $SL_2(\R)$, we have $n=pq'-p'q$. Since $I_{r,r'}$ does not depend on $(x,y)$,   we might as well assume $(x,y)=(0,1)$, that is, $X_s=St(2s-1)$. Observe (in Figure \ref{upper_right}) that   there exists a saddle connection $\gamma_1$ with direction $(1,n)$, whose homology class is 
$\left[ e_1 \right]+ n\left[ f_1 \right]$, so $\mbox{Int}(e_1, \gamma_1) = n$, hence $I_{r,r'}=1$.

\textbf{Second case} : $V(r)=1$, that is, $V(p,q)=(1,1)$. Since $V(r') \in \Z$, there exists $n \in \Z$ such that $V(p',q') = (-n,1)$. Since $I_{r,r'}$ does not depend on $(x,y)$,   we might as well assume $(x,y)=(-1,1)$. The surface $X_s(-1,1)$ has a $(2s-1)$-square template depicted in Figure \ref{un-cylindre}, and there exists a saddle connection $\gamma_1$ with direction $(-n,1)$, whose homology class is 
$-n\left[ e_s \right]+ \left[ g_s \right]$, so $\mbox{Int}(g_s,\gamma_1) = n$, hence $I_{r,r'}=1$.

\end{proof}

\begin{figure}
\begin{center}
\begin{tikzpicture}[scale=2]
\draw (1,0) -- (3,0);
\draw (2,0) -- (2,2);
\draw (3,0) -- (3,2);
\draw (1,0) -- (1,1);
\draw [color=red]  (1,0) -- (4/3,1);
\draw [color=red]  (4/3,0) -- (5/3, 1);
\draw [color=red]  (5/3,0) -- (2,1);
\draw (1,1) -- (3,1);
\draw (2,0) -- (2,1);
\draw (2,2) -- (3,2);
\draw (1.1,.5) node[anchor=north west, color=red] {$\gamma_1$};
\draw (2.5,0) node[anchor=north west] {$e_2$};
\draw (1.5,0) node[anchor=north west] {$e_1$};
\draw (2.5,1.3) node[anchor=north west] {$e'_2$};
\end{tikzpicture}
\caption{the saddle  connection $\gamma_1$, with $s=2$ and $n=3$, in the two-cylinder case}\label{upper_right}
\end{center}
\end{figure}
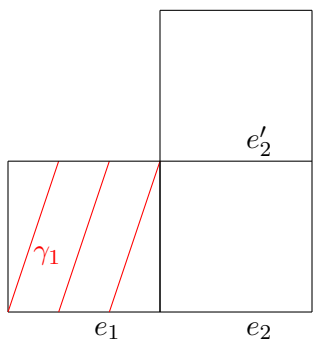

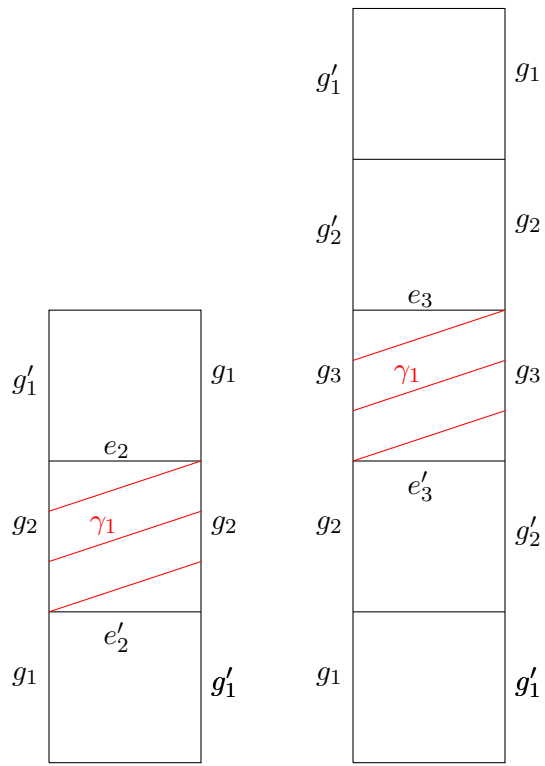
\begin{figure}
\begin{center}
\begin{tikzpicture}[scale=2]
\draw (0,0) -- (1,0)--(1,3)--(0,3)--(0,0);
\draw (0,1) -- (1,1);
\draw (0,2) -- (1,2);
\draw [color=red]  (0,1) -- (1,4/3);
\draw [color=red]  (0,4/3) -- (1,5/3);
\draw [color=red]  (0,5/3) -- (1,2);
\draw (0,.7) node[anchor=north east] {$g_1$};
\draw (0,1.7) node[anchor=north east] {$g_2$};
\draw (0,2.7) node[anchor=north east] {$g'_1$};
\draw (1,.7) node[anchor=north west] {$g'_1$};
\draw (1,.7) node[anchor=north west] {$g'_1$};
\draw (1,1.7) node[anchor=north west] {$g_2$};
\draw (1,2.7) node[anchor=north west] {$g_1$};
\draw (.2,1.7) node[anchor=north west, color=red] {$\gamma_1$};
\draw (0.6,1) node[anchor=north east] {$e'_2$};
\draw (0.6,2.2) node[anchor=north east] {$e_2$};

\draw (2,0) -- (3,0)--(3,5)--(2,5)--(2,0);
\draw (2,1) -- (3,1);
\draw (2,2) -- (3,2);
\draw (2,3) -- (3,3);
\draw (2,4) -- (3,4);
\draw [color=red]  (2,2) -- (3,7/3);
\draw [color=red]  (2,7/3) -- (3,8/3);
\draw [color=red]  (2,8/3) -- (3,3);
\draw (2,.7) node[anchor=north east] {$g_1$};
\draw (2,1.7) node[anchor=north east] {$g_2$};
\draw (2,2.7) node[anchor=north east] {$g_3$};
\draw (3,.7) node[anchor=north west] {$g'_1$};
\draw (3,.7) node[anchor=north west] {$g'_1$};
\draw (3,1.7) node[anchor=north west] {$g'_2$};
\draw (3,2.7) node[anchor=north west] {$g_3$};
\draw (3,3.7) node[anchor=north west] {$g_2$};
\draw (3,4.7) node[anchor=north west] {$g_1$};
\draw (2,3.7) node[anchor=north east] {$g'_2$};
\draw (2,4.7) node[anchor=north east] {$g'_1$};
\draw (2.2,2.7) node[anchor=north west, color=red] {$\gamma_1$};
\draw (2.6,2) node[anchor=north east] {$e'_3$};
\draw (2.6,3.2) node[anchor=north east] {$e_3$};

\end{tikzpicture}
\caption{the saddle  connection $\gamma_1$, for $s=2$, and $s=3$, and $n=3$, in the one-cylinder case}\label{un-cylindre}
\end{center}
\end{figure}

\begin{lemma} \label{99_100}
For $r,r' \in \Q \cup \{\infty\}$, with $(r,r') \not\in \mbox{End}(\mathcal{Z})$, we have $I_{r,r'} < 9/10$.
\end{lemma}
\begin{proof}
Let $p,q,p',q' \in \Z$ be such that $p \wedge q = p' \wedge q' = 1$, $p/q=r$, $p'/q'=r'$. 

\textbf{First case} : $p,q,p',q'$ are not all odd. 

Multiplying $(p,q)$ and $(p',q')$ by some element of the Veech group, and swapping $(p,q)$ and $(p',q')$ if necessary, we may assume $(p',q')=(1,0)$, that is, $r'= \infty$,  and $|p| < |q|$, so $pq'-p'q= q$. Furthermore by the symmetries  of $St(2s-1)$, we only need consider the case when $0 < p < q$. Note that we then have $q \geq 2$, otherwise $p$ is $0$  and in that case we have $(r,r') \in \mbox{End}(\mathcal{Z})$. Let $\gamma$ be a saddle connection with direction $(p,q)$. We want to show that 
\begin{align*}
\mbox{Int} (e_i, \gamma) &= \mbox{Int} (e'_i, \gamma) < \frac{9q}{10}  \mbox{ for } i = 2, \ldots, s\\
 \mbox{Int} (e_1, \gamma) & < \frac{9q}{10}.
 \end{align*}
 The projection $\Pi$ restricted to $\gamma$ is injective by Lemma \ref{1to1}, so the $q-1$ intersections of $(p,q)$ with $(1,0)$ in $\T^2 \setminus \{(0,0)\}$  lift to $q-1$ intersections of $\gamma$ with the singular cycle $e_1+e_2+e'_2+\ldots e_s+e'_s$.  Let us consider the sequence $S_{int}$ of all intersections of $\gamma$ with $e_1, e_i, e'_i$, for $i=2, \ldots, s$, in cyclical order along $\gamma$, with the intersection at the conical point $S$ set apart. Denote $\#S_{int} (e_i)$ 
 (resp. $\#S_{int} (e'_i)$) the number of $e_i$'s (resp. $e'_i$'s) in the sequence $S_{int}$. We have 
 \[
 \mbox{Int} (e, \gamma) \leq \#S_{int} (e)+1 \mbox{ for } e=e_1, \ldots, e_s, \mbox{ and } e=e'_2, \ldots, e'_s,
\] 
where the $+1$ accounts for the intersection at $S$. 
 Observe that in the sequence $S_{int}$ there are never two consecutive $e_i$ or $e'_i$, for $i=2, \ldots, s$~: each $e'_i$ is followed by $e_i$ or $e'_{i+1}$, and each $e_i$ is followed by $e'_i$ or $e_{i-1}$.
 So the proportion of $e_i$ or $e'_i$ in $S_{int}$, for $i=2, \ldots, s$ is at most $1/2$ (only possible when $s=2$).  Therefore, since $q \geq 2$, 
\[
    \mbox{Int} (e_i, \gamma) =  \mbox{Int} (e'_i, \gamma)  \leq \frac{q-1}{2}+1 \leq \frac{3}{4}q .
\] 
 Intersections with $e_1$ must be treated separately, with  a case-by-case analysis :
 
 \textbf{Case 1.1}: $p < q < 2p$. Observe that each block of $e_1$ has length at most $2$, and  is followed by at least an $e'_2$, so  the proportion of $e_1$ in the sequence $S_{int}$ is at most $2/3$, so 
 \[
  \mbox{Int} (e_1, \gamma) \leq \frac{2(q-1)}{3} +1 \leq \frac{9q}{10}.
  \]  
  \textbf{Case 1.2}: $2p < q < 3p$ (see Figure \ref{pq37}). Then a block of $e_1$ has length at most $3$, and is followed by at least an $e'_2$ and an $e_2$, so as previously, the proportion of $e_1$ in the sequence $S_{int}$ is at most $3/5$, so 
 \[
  \mbox{Int} (e_1, \gamma) \leq \frac{3(q-1)}{5} +1 \leq \frac{9q}{10}.
  \]
 \textbf{Case 1.3}: $3p < q $.  Intersections with $e_1$  come in blocks of length at most  $\lceil q/p \rceil$, because two consecutive (along $\gamma$) intersections are exactly $p/q$ apart along $e_1$ (see Figure \ref{pq37}). Each block of $e_1$'s is followed by a block of $e'_2 e_2$ (this is where we use the fact that $(r,r') \not\in \mbox{End}(\mathcal{Z})$), of length at least $\lfloor q/2p \rfloor$. Recall that $2\lfloor q/2p \rfloor \geq \lceil q/p \rceil -2$. So the proportion of $e_1$ in $S_{int}$ is at most 
 \[
 \frac{\lceil q/p \rceil}{\lceil q/p \rceil + 2\lfloor q/2p \rfloor} \leq \frac{\lceil q/p \rceil}{2\lceil q/p \rceil -2} \leq \frac{3}{4}
 \]
 where the last inequality stands because $3 < q/p $. Then 
 \[
  \mbox{Int} (e_1, \gamma) \leq \frac{3(q-1)}{4} +1 \leq \frac{9q}{10}.
  \]
  This finishes the first case.
\begin{figure}
\begin{center}
\begin{tikzpicture}[line cap=round,line join=round,>=triangle 45,x=1cm,y=1cm, scale=1/2]
\clip(-5.32943931065326,-3.1162154079788844) rectangle (22.250064962662353,15);
\draw [line width=1pt] (0,0)-- (14,0)-- (14,14)-- (7,14)-- (7,7)-- (0,7)--(0,0);

\draw [line width=1pt,color=ffqqqq] (13,0)-- (14,7/3);
\draw [line width=1pt,color=ffqqqq] (0,2.3333333333333335)-- (2,7);
\draw [line width=1pt,color=ffqqqq] (2,0)-- (5,7);
\draw [line width=1pt,color=ffqqqq] (5,0)-- (11,14);
\draw [line width=1pt,color=ffqqqq] (11,0)-- (14,7);
\draw [line width=1pt,color=red] (7,0)-- (13,14);

\draw [<->, line width=1pt] (7,-1)-- (13,-1);
\draw [<->, line width=1pt] (2,-1)-- (5,-1);

\draw [fill=black] (0,0) circle (2pt);
\draw [fill=black] (14,14) circle (2.5pt);
\draw [fill=black] (14,7) circle (2.5pt);
\draw [fill=black] (7,7) circle (2.5pt);

\draw [fill=black] (7,14) circle (2.5pt);
\draw [fill=black] (7,0) circle (2.5pt);
\draw [fill=black] (0,7) circle (2.5pt);
\draw[color=black] (3.5,-1.8) node {$\frac{p}{q}$};
\draw[color=black] (10,-1.8) node {$\frac{2p}{q}$};
\draw[color=red] (9,3) node {$\gamma$};

\end{tikzpicture}
\caption{$(p,q)=(3,7)$}\label{pq37}
\end{center}
\end{figure}

\begin{remark} 
There is numerical evidence that 
%$(r,r’) \not \in \Z^2$ and $r'$ is in the orbit of infinity under the Veech group, 
 $I_{r,\infty} \leq 2/3$ for all $r \in \Q$ such that $(r,\infty) \not\in \mbox{End}(\mathcal{Z})$,  except when $r=3/7$, in which case $I_{r,\infty} =5/7$.
\end{remark}

\textbf{Second case} : $p,q,p',q'$ are all odd. 
Multiplying $(p,q)$ and $(p',q')$ by some element of the Veech group, and swapping $(p,q)$ and $(p',q')$ if necessary, we may assume $(p',q')=(1,1)$, that is, $r'= 1$, so $pq'-p'q= p-q$. Again multiplying by an element of the Veech group which stabilizes $(1,1)$, we may assume $(p,q)=(p-q)(1,0) +q(1,1)$ satisfies $|p-q|>|q|$.  Furthermore by the symmetries  of $St(2s-1)$, we only need consider the case when 
$0 < q < p-q$. Note that we then have $p-q \geq 2$, otherwise $p=q$  and in that case we have $(r,r') \in \mbox{End}(\mathcal{Z})$. Let 
$\gamma$ be a saddle connection with direction $(p,q)$. We want to show that 
\begin{align*}
\mbox{Int} (\gamma, g_i) & < \frac{9(p-q)}{10}  \mbox{ for } i = 1, \ldots, s\\
 \mbox{Int} ( \gamma, g'_i) & < \frac{9(p-q)}{10} \mbox{ for } i = 1, \ldots, s-1.
 \end{align*}
 The projection $\Pi$ restricted to $\gamma$ is injective by Lemma \ref{1to1}, so the $p-q-1$ intersections of $(p,q)$ with $(1,1)$ in $\T^2 \setminus \{(0,0)\}$  lift to $p-q-1$ intersections of $\gamma$ with the singular cycle $g_1+g'_1+g_2+g'_2+\ldots +g_s$.  Let us consider the sequence $S_{int}$ of all intersections of $\gamma$ with $g_1, g'_1, \ldots, g_s$,  in cyclical order along $\gamma$, with the intersection at the conical point $S$ set apart. Denote $\#S_{int} (g_i)$ 
 (resp. $\#S_{int} (g'_i)$) the number of $g_i$'s (resp. $g'_i$'s) in the sequence $S_{int}$. 
We have 
 \[
 \mbox{Int} (g, \gamma) \leq \#S_{int} (g)+1 \mbox{ for } g=g_1, \ldots, g_s, \mbox{ and } g=g'_1, \ldots, g'_{s-1},
\] 
where the $+1$ accounts for the intersection at $S$. 
 Observe that in the sequence $S_{int}$ there are never two consecutive $g_i$ , for $i=1, \ldots, s-1$, or $g'_i$, for $i=2, \ldots, s-1$ : each $g'_i$  ($2 \leq i \leq s-1$) is followed by $g_i$ or $g_{i-1}$,
  and each $g_i$ ($1 \leq i \leq s-1$) is followed by $g'_i$ or $g'_{i+1}$  (see Figure \ref{un-cylindre}).
 So the proportion of $g_i$ or $g'_i$ in $S_{int}$, for $i=2, \ldots, s$ is at most $1/2$ (only possible when $s=2$).  Therefore, since $p-q \geq 2$, 
\[
    \mbox{Int} (g, \gamma)  \leq \frac{p-q-1}{2}+1 \leq \frac{3}{4}(p-q)\mbox{ for } g=g_1, \ldots, g_{s-1}, \mbox{ or } g=g'_2, \ldots, g'_{s-1} .
\] 

 Intersections with $g'_1$ and $g_s$ must be treated separately, with  a case-by-case analysis. First we deal with  $g_s$ : 
 
 \textbf{Case 2.1}: $q < p-q < 2q$. Observe that each block of  $g_s$ has length at most $2$, and  is followed by at least a $g_{s-1}$, so  the proportion of $g_s$ in the sequence $S_{int}$ is at most $2/3$, whence 
 $\mbox{Int} (g_s, \gamma) \leq  \frac{9(p-q)}{10}$.

 \textbf{Case 2.2}: $2q < p-q < 3q$. Then a block of $g_s$ has length at most $3$, and is followed by at least a $g_{s-1}$ and a $g'_{s-1}$, so so  the proportion of $g_s$ in the sequence $S_{int}$ is at most $3/5$, whence 
 $\mbox{Int} (g_s, \gamma) \leq  \frac{9(p-q)}{10}$.
 
  \textbf{Case 2.3}: $3q < p-q $.  Intersections with $g_s$  come in blocks of length at most  $\lceil (p-q)/q \rceil$, because two consecutive (along $\gamma$) intersections are exactly $q/(p-q)$ apart along $g_s$ (see Figure \ref{un-cylindre}). Each block of $g_s$'s is followed by a block of $g_{s-1} g'_{s-1}$, of length at least $\lfloor (p-q)/2q \rfloor$. Recall that $2\lfloor (p-q)/2q \rfloor \geq \lceil (p-q)/q \rceil -2$. So the proportion of $g_s$ in $S_{int}$ is at most 
 \[
 \frac{\lceil (p-q)/q \rceil}{\lceil (p-q)/q \rceil + 2\lfloor (p-q)/2q \rfloor} \leq \frac{\lceil (p-q)/q \rceil}{2\lceil (p-q)/q \rceil -2} \leq \frac{3}{4}
 \]
 where the last inequality stands because $3 < (p-q)/q $. Again, $\mbox{Int} (g_s, \gamma)  \leq \frac{9(p-q)}{10}$.
 
Now we deal with $g'_1$. In that case it is more convenient to write $(p,q)=(p-q)(2,1)+(2q-p)(1,1)$. Note that $2q-p < 0$ because $p-q >q$.

 \textbf{Case 2.4}: $p-q < 2|2q-p|$, that is, $2q < p-q $. Observe that each block of  $g'_1$ has length at most $2$, and  is followed by at least a $g_{1}$, so  the proportion of $g'_1$ in the sequence $S_{int}$ is at most $2/3$, whence 
 $\mbox{Int} (g'_1, \gamma) \leq  \frac{9(p-q)}{10}$.

 \textbf{Case 2.5}: $2|2q-p|<p-q < 3|2q-p|$, that is, $3q/2 < p-q < 3q$. Then a block of $g'_1$ has length at most $3$, and is followed by at least a $g_{1}$ and a $g'_{2}$, so so  the proportion of $g_s$ in the sequence $S_{int}$ is at most $3/5$, whence 
 $\mbox{Int} (g_s, \gamma) \leq  \frac{9(p-q)}{10}$.
 
  \textbf{Case 2.6}: $p-q > 3|2q-p|$, that is, $q < p-q < 3q/2$.  Intersections with $g'_1$  come in blocks of length at most  $\lceil (p-q)/|2q-p| \rceil$, because two consecutive (along $\gamma$) intersections are exactly $|2q-p|/(p-q)$ apart along $g'_1$ (see Figure \ref{un-cylindre}). Each block of $g'_1$'s is followed by a block of $g'_{2} g_{1}$, of length at least $\lfloor (p-q)/2|2q-p| \rfloor$. Recall that $2\lfloor (p-q)/2|2q-p| \rfloor \geq \lceil (p-q)/|2q-p| \rceil -2$. So the proportion of $g'_1$ in $S_{int}$ is at most 
 \[
 \frac{\lceil (p-q)/|2q-p| \rceil}{\lceil (p-q)/|2q-p| \rceil + 2\lfloor (p-q)/2|2q-p| \rfloor} \leq \frac{\lceil (p-q)/|2q-p| \rceil}{2\lceil (p-q)/|2q-p| \rceil -2} \leq \frac{3}{4}
 \]
 where the last inequality stands because $3 < (p-q)/|2q-p| $. Again,  $ \mbox{Int} (g'_1, \gamma)  \leq \frac{9(p-q)}{10}$.

\end{proof}

%%%%
\section{Proof of Theorem \ref{disqueL(2,2)}}\label{results L22}
\begin{lemma}\label{majorationPetitCarreaux}
Take $s \in \N^*$ and $X \in \mathcal{H}(2s-2)$. Assume   $X$ is an $n$-fold ramified Riemannian cover of a  flat torus. Then $\mbox{KVol}(X)\leq n$, unless there exists a pair of closed geodesics $\alpha$ and $\beta$ on $X$, which have the same direction, non-zero intersection, and such the product of their lengths is $<\mbox{Vol}(X)/n$.
\end{lemma}
\begin{proof}
Assume, up to isometry and re-scaling, that  the flat torus which $X$ covers is $\R^2 / (1,0)\Z \oplus (a,b)\Z$, with $|a| \leq 1/2$ and $a^2+b^2 \geq 1$. Then any closed geodesic on the torus is at least $1$ long, and the volume of the torus is $b$, so the volume of $X$ is $nb$. 

Let  $\alpha$ and $\beta$ be simple closed geodesics in  $X$. Since  $X$ has but one conical point, both $\alpha$ and  $\beta$  have a well-defined direction, say  $\frac{p}{q}$ and $\frac{p'}{q'}$ in irreducible terms.

\begin{remark} If  $X$ had several conical points, $\alpha$ or $\beta$ could be made up of several saddle connections of distinct directions. 
\end{remark}

Let $\Pi : X \longrightarrow \T^2$ be the  projection of the ramified cover. 

Then $\Pi$  maps $\alpha$ (resp. $\beta$) to a closed geodesic of the torus, with homology class $(p,q)$ (resp. $(p',q')$). 

\textbf{First case} : $pq'-p'q \neq 0$.

Observe that  $\Pi(\alpha)$ and $\Pi(\beta)$  intersect exactly  $|pq'-p'q|$ times, and their intersections are equidistributed along $\Pi(\alpha)$, since the first return map to  $\Pi(\alpha)$ of the linear flow with  direction $(p',q')$, is a  rotation of $\Pi(\alpha)$. Therefore, two consecutive  intersections are exactly $\sqrt{(p+aq)^2+b^2q^2}/|pq'-p'q|$ apart  along  $\Pi(\alpha)$.

Thus, two consecutive intersections of  $\alpha$ and $\beta$ cannot be  less than $\frac{\sqrt{(p+aq)^2+b^2q^2}}{|pq'-p'q|}$ apart along  $\alpha$, since the restriction to  $\alpha$ of $\Pi$ is a local isometry. Hence denoting  $l(\alpha)$ the length of  $\alpha$, 
\[
\mbox{Int}(\alpha, \beta) \leq \frac{ |pq'-p'q|}{\sqrt{(p+aq)^2+b^2q^2} }l(\alpha).
\]

Besides, since  $\Pi$ is 1-Lipschitz, we have $l(\beta) \geq \sqrt{(p'+aq')^2+b^2q'^2}$, whence
\[
\frac{\mbox{Int}(\alpha, \beta) }{l(\alpha)l(\beta)} \leq \frac{ |pq'-p'q|}{\sqrt{(p+aq)^2+b^2q^2}\ \sqrt{(p'+aq')^2+b^2q'^2}} \leq K(\T^2)=\frac{1}{b}
\]
where the last equality stems from \cite{MM}.

\textbf{Second case} : $pq'-p'q= 0$. Then since  $\alpha$ et $\beta$ are  simple, we have $p=p'$, $q=q'$.
So the closed curves $\alpha$ and $\beta$, if they are distinct,  cannot meet anywhere but at the conical point. Hence their algebraic  intersection is $0$ or $\pm 1$. Therefore
\[
\frac{\mbox{Int}(\alpha, \beta) }{l(\alpha)l(\beta)} \leq \frac{ 1}{(p+aq)^2+b^2q^2} \leq 1.
\]
So we have $\mbox{KVol}(X)\leq  \frac{1}{b}nb=n$, unless there exists a pair of closed geodesics $\alpha$ and $\beta$ on $X$, which have the same direction (that is, $(p,q)=(p',q')$), non-zero intersection, and such that the product of their lengths is $< b$.
\end{proof}

\begin{corollary} \label{corL22}
We have $\mbox{KVol}(St(2s-1))=2s-1$.
\end{corollary}
\begin{proof}
Lemma \ref{majorationPetitCarreaux} and the facts  that any closed curve on $St(2s-1)$ is at least $1$ long, and that $\mbox{Vol}(St(2s-1))=2s-1$, entail that $\mbox{KVol}(St(2s-1)) \leq 2s-1$. On the other hand, we have $\mbox{Int}(e_1, f_2)=1$, $l(e_1)=l(f_2)=1$.
\end{proof}

Now we want to know for which elements $X=X_s(x,y)$ of $\mathcal{T}_s$ we  have $\mbox{KVol}(X_s(x,y))>2s-1$. 
\begin{lemma}\label{>3}
For $(x,y) \in V_{\pm 1}  \cap  \mathcal{D}$,  where   $V_{\pm 1}= \{ (x,y)  \ : \  (x \pm 1)^2 +(y -1/2)^2 < 1/4 \}$, we have $\mbox{KVol}(X_s(x,y))>2s-1$. For any other $(x,y)$ in $\mathcal{D}$, we have 
$ \mbox{KVol}(X_s(x,y)) \leq 2s-1$. Furthermore, $ \mbox{KVol}(X_s(x,y)) $ goes to infinity when $y$ goes to zero, and goes to $2s-1$ when $y$ goes to $\infty$ while $(x,y)$ remains in $\mathcal{D}$.
\end{lemma}
\begin{proof}
By Lemma \ref{majorationPetitCarreaux} and Subsubsection \ref{subsecSaddle connections and intersection},  finding which elements $X_s(x,y)$ of $\mathcal{T}_s$   have $\mbox{KVol}(X_s(x,y))>2s-1$ amounts to finding all $(x,y)$ in the fundamental domain $\mathcal{D}$ such that there exist $p,q \in \Z$, coprime and both odd, such that  
\begin{align}
\frac{y}{(p+qx)^2+(qy)^2}  & >  1 \\
\Longleftrightarrow (x+\frac{p}{q})^2 + (y-\frac{1}{2q^2})^2 & < \frac{1}{4q^4}
\end{align}
that is, $(x,y)$ lies inside the open disk $D(p,q)$ of radius $1/2q^2$, centered at $(-p/q, 1/2q^2)$. 
As we shall see this only happens when $|p|=|q|=1$. Indeed if $|p|>|q|$, the center of $D(p,q)$ lies at least $1/q$ apart from the vertical boundaries of $\mathcal{D}$, and since the radius of $D(p,q)$ is $<1/q$, the whole $D(p,q)$ lies outside of $\mathcal{D}$. If $|p|<|q|$, the center of $D(p,q)$ lies below the half-circle $x^2+y^2=1$, and the distance between the center of $D(p,q)$ and the half-circle $x^2+y^2=1$ is greater than the radius of $D(p,q)$, because
\[
1- \sqrt{\frac{p^2}{q^2}+\frac{1}{4q^4}} > \frac{1}{2q^2},
\]
so the whole $D(p,q)$ lies below the half-circle $x^2+y^2=1$. Since $p$ and $q$ are coprime, the only remaining possibility is $p=\pm 1$, $q=1$. So for any $X_s(x,y)$ in $\mathcal{T}_s$, we have $\mbox{KVol}(X_s(x,y))\leq2s-1$, unless $(x,y)$ lies in the horocyclic neighborhood $V_{\pm 1}$ of the lower cusp of $\mathcal{D}$, bounded by the dotted circles depicted in Figure \ref{horocycle}.

For $(x,y)$ inside $V_{\pm 1}$, we have 
\[
\mbox{KVol}(X_s(x,y)) \geq  \max \left( \frac{(2s-1)y}{\sqrt{(1+x)^2+y^2}} , \frac{(2s-1)y}{\sqrt{(1-x)^2+y^2}}  \right) >2s-1
\]
and
\[
\lim_{y \rightarrow 0} \mbox{KVol}(X_s(x,y)) = +\infty.
\]
Since $\mbox{KVol}(X_s(1,y))=2s-1$, $\mbox{KVol}$ is continuous as a function of $(x,y)$, and  every $(x,y)$ in $\mathcal{D}$ is within distance $1/y$ of $(1,y)$, it follows that  $\mbox{KVol}(X_s(x,y))$ tends to $2s-1$ when $y$ goes to $\infty$.
\end{proof}

For $p,q \in \Z$, $p \wedge q = 1$, setting $r=p/q$, and for $(x,y) \in \Hy^2$, define 
\[
J_r (x,y)= \frac{y}{(p+qx)^2+(qy)^2} .
\]

Now, thanks to Lemmata  \ref{majorationPetitCarreaux} and \ref{>3}, we can give a more precise version of Equation~(\ref{etoile}) :
\begin{equation}
\forall (x,y) \in \mathcal{D}, \      \mbox{KVol}(X_s(x,y))= (2s-1) \max \large\{ J_1 (x,y), J_{-1} (x,y), \sup_{r \neq r' \in \Q} I_{r,r'} K_{r,r'} (x,y)\large\}.
\end{equation}

\begin{lemma}
For  $(x,y) \in \mathcal{D}$, for all $(r,r') \in \mbox{End}(\mathcal{Z})$, we have 
\[
K_{r,r'} (x,y) \geq  \sqrt{\frac{143}{144}}
\]
and equality occurs if and only if $(x,y)=(\pm 9/14, \sqrt{143}/14)$.
\end{lemma}
\begin{proof}
	By Lemma \ref{Irr'=1}, we have $I_{r,r'}=1$ whenever $(r,r') \in \mbox{End}(\mathcal{Z})$, and by Lemma \ref{99_100}, we have $I_{r,r'} \leq 9/10 < \sqrt{\frac{143}{144}}$ whenever 
	$(r,r') \not\in \mbox{End}(\mathcal{Z})$.
	
	Elementary calculations show that the bissectors of the triangle T delimited by the geodesics  $\gamma_{-1,1}$, $\gamma_{-2,1}$ and $\gamma_{0,2}$ intersect at the point $ k=(\frac{9}{14}, \frac{\sqrt{143}}{14})$. For each $\gamma_{r,r'}$ different from the sides of T, $\gamma_{r,r'}$ does not pass through the interior of T, therefore the distance between the point  $ k$ and $\gamma_{r,r'}$ is greater than the distance between $ k$ and the sides of the triangle T.  Thus, for each $\gamma_{r,r'}$ different from  $\gamma_{-1,1}$, $\gamma_{-2,1}$ and $\gamma_{0,2}$, we have $ \theta_{r,r'}(k)> \theta_{-1,1}(k)$, so $\cos \theta_{r,r'}(k)<\cos \theta_{-1,1}(k)$,  which entails:
\[
\sup_{(r,r')} K_{r,r'} (k)=K_{-1,1}(k)=\sqrt{\frac{143}{144}}.
\]
	
	In all that follows, we call $V_{r,r'}$ the banana neighbourhood  of the geodesic $\gamma_{r,r'}$ such that 
\[
\theta_{r,r'}(x,y)= \theta_{-1,1}(\frac{9}{14}, \frac{\sqrt{143}}{14} ).
\]
Let 
\[
C_{n}=C((\frac{1-n}{2},-\frac{1}{\sqrt{143}}\frac{1+n}{2}),\frac{1+n}{2}\sqrt{\frac{144}{143}})
\] 
and
\[
 C^{n}=C((\frac{1-n}{2},\frac{1}{\sqrt{143}}\frac{1+n}{2}),\frac{1+n}{2}\sqrt{\frac{144}{143}})
 \]
be the circles with centre respectively 
\[
(\frac{1-n}{2},-\frac{1}{\sqrt{143}}\frac{1+n}{2}) \mbox{ and } (\frac{1-n}{2},\frac{1}{\sqrt{143}}\frac{1+n}{2})
\]
  and the same radius $\frac{1+n}{2}\sqrt{\frac{144}{143}}$, $ n \geq 1$.	The banana neighbourhood $V_{-n,1}$, for $ n \geq 1$, is delimited by the maximal arcs of $C_{n}$ and $C^{n}$, respectively, which are contained in $\Hy^2$ ; and the banana neighbourhood $V_{\frac{1}{2},\infty}$ is delimited by the two straight lines of  equation $y=\sqrt{143}x+1/2$ and $y=-\sqrt{143}x+1/2$, respectively. 
  
  Let $\mathcal{A}$ be the region delimited by  $ x=1/2$, $x=1$ and $\gamma_{-1,1}$. For all $n \geq 3$,  $C_{n}\cap \mathcal{A}$, $C^{n}\cap \mathcal{A}$ is contained  in the interior of, respectively, $V_{-n+1,1}$, and $V_{-n-1,1}$ ;  and $C^{2} \cap \mathcal{A}$ is contained  in the interior of $V_{-3,1}$ (see Figure \ref{covering}).

 Also, the interiors of  $V_{0,2}$, $V_{-1,1}$ and $V_{\frac{1}{2},\infty}$ cover the region delimited by  $C_{2}$, $\gamma_{-1,1}$ and  $\gamma_{\frac{1}{2},\infty}$  except the point $k$ (see Figure \ref{covering}), so the interior of  $V_{0,2}$,  $V_{\frac{1}{2},\infty}$ and $V_{-n,1}$,  $ n \geq 1$, cover all  the region $\mathcal{A}$ except the point $k$.
	
Thus for  $(x,y) \in \mathcal{A}$, $(x,y)\ne k$, we have $K_{r,r'} (x,y) >  \sqrt{\frac{143}{144}}$. By symmetry with respect to $x = 1/2$, and afterwards by $x = 0$,  we then deduce the lemma.
\end{proof}

\begin{figure}
	\begin{center}
		\includegraphics[scale=0.3]{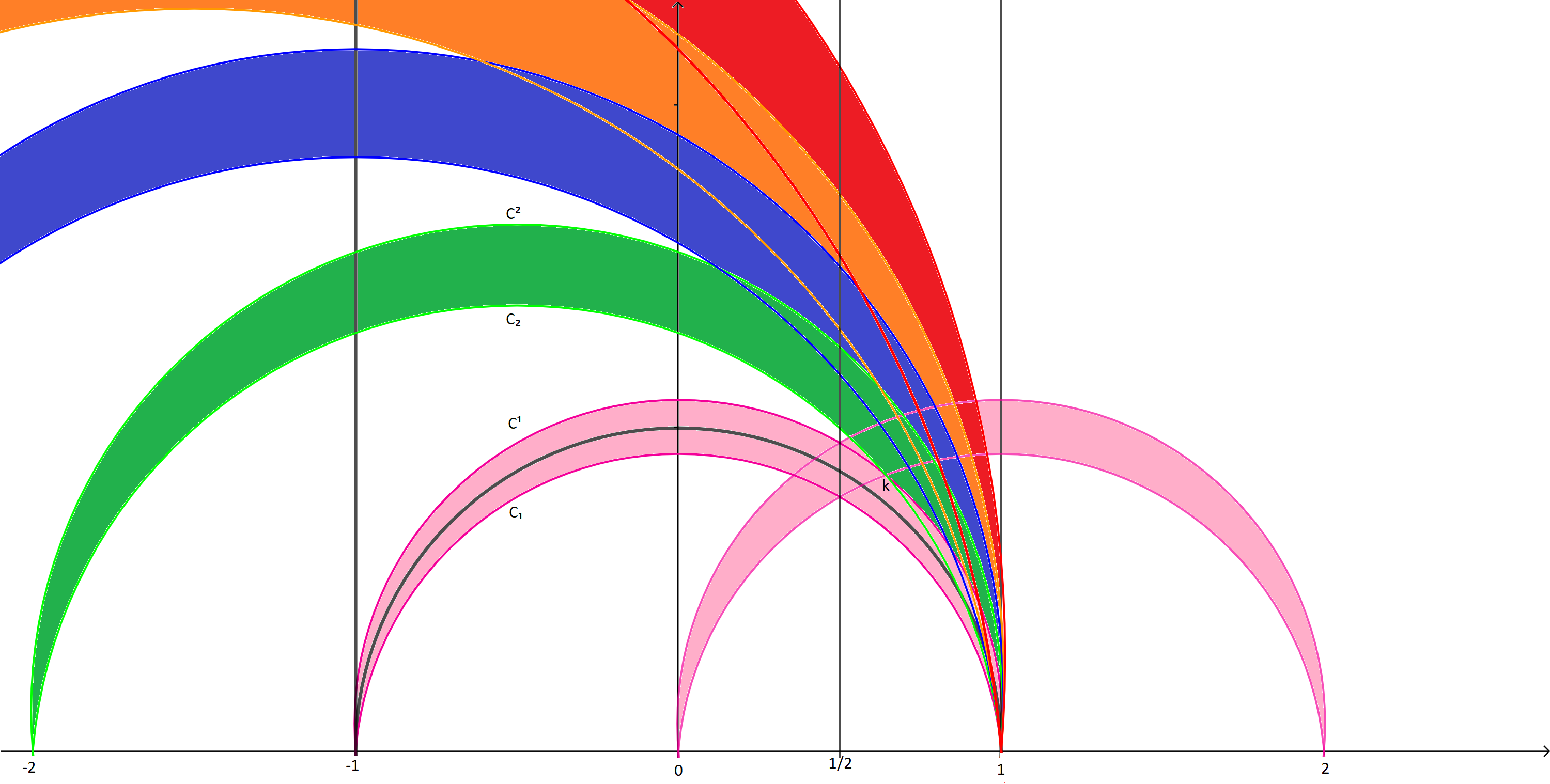}
		\caption{covering by banana neighbourhoods}
		\label{covering}
		
\end{center}
\end{figure}

We are now set to finish the proof of Theorem \ref{disqueL(2,2)}.
It is straightforward to check that 
\[
 \max \large\{ J_1 ( \frac{9}{14}, \frac{\sqrt{143}}{14}), J_{-1} ( \frac{9}{14}, \frac{\sqrt{143}}{14}) \large\}
 = \max \large\{ J_1 ( -\frac{9}{14}, \frac{\sqrt{143}}{14}), J_{-1} ( -\frac{9}{14}, \frac{\sqrt{143}}{14} ) \large\}
 = \sqrt{\frac{143}{144}}
\]
so we have 
\[
\mbox{KVol}(X_s(\pm \frac{9}{14}, \frac{\sqrt{143}}{14}))= (2s-1)\sqrt{\frac{143}{144}},
\]
and for every $(x,y) \in \mathcal{D}$, we have $\mbox{KVol}(X_s(x,y)) \geq (2s-1)\sqrt{\frac{143}{144}}$, with equality if and only if $(x,y)=(\pm 9/14, \sqrt{143}/14)$.

\bigskip

\noindent \textbf{Adresses} : 

\noindent Smaïl Cheboui :  USTHB, Facult\'e de Math\'ematiques, Laboratoire de Syst\`emes Dynamiques, 16111 El-Alia BabEzzouar - Alger, Alg\'erie, 

\noindent email : scheboui@usthb.dz

\noindent Arezki Kessi :       USTHB, Facult\'e de Math\'ematiques, Laboratoire de Syst\`emes Dynamiques, 16111 El-Alia BabEzzouar - Alger, Alg\'erie, 

\noindent email: akessi@usthb.dz

\noindent Daniel Massart : Institut Montpelliérain Alexander Grothendieck, CNRS, Universit\'e de  Montpellier, France, 

\noindent email : daniel.massart@umontpellier.fr (corresponding author)

\end{document}